\documentclass{amsart}
\usepackage[utf8]{inputenc}

\usepackage[foot]{amsaddr}
\usepackage{thm-restate}
\usepackage{amsthm}
\usepackage{amsmath,amsfonts,amssymb}
\usepackage{thmtools, thm-restate}
\usepackage{parskip}
\usepackage{enumitem}
\newcommand{\gbar}{\ensuremath{\overline{G}}}
\newcommand{\gtilde}{\ensuremath{\widetilde{G}}}

\renewcommand{\leq}{\leqslant}
\renewcommand{\geq}{\geqslant}

\newcommand{\floor}[1]{\ensuremath{\lfloor{} #1 \rfloor{}}}
\newcommand{\ceil}[1]{\ensuremath{\lceil{} #1 \rceil{}}}
\newcommand{\func}[3]{\ensuremath{#1 : #2\rightarrow #3}}%

\newcommand{\emk}{\ensuremath{\mathcal{E}_{m,k}}}

\newtheorem{theorem}{Theorem}[section]
\newtheorem{lemma}[theorem]{Lemma}
\newtheorem{cor}[theorem]{Corollary}

\newtheorem{clm}{Claim}
\newtheorem{remark}[theorem]{Remark}
\theoremstyle{definition}
\newtheorem{dfn}[theorem]{Definition}
\newtheorem{definition}[theorem]{Definition}
\newcommand{\parts}{\ensuremath{\vec{\sigma}}}
\newcommand{\admat}{\ensuremath{\mathbf{D}}}

\title{The extremal function for structured sparse minors}
\author{Matthew Wales}
\date{\today{}}
\address{DPMMS, University of Cambridge, CB3 0WB, UK}

\begin{document}

\maketitle
\begin{abstract}
    Let $c(H)$ be the smallest value for which $e(G)/|G|\geq c(H)$ implies $H$ is a minor of $G$. We show a new upper bound on $c(H)$, which improves previous bounds for graphs with a vertex partition where some pairs of parts have many more edges than others --- for instance a complete bipartite graph with a small number of edges placed inside one class. We also show a tight matching lower bound for almost all such graphs. We apply these results to show $c(K_{ft/\log t,t}) = (0.638\dotsc+o_{f}(1))t\sqrt{f}$, for $f = o(\log t) = \omega(1)$.
    
\end{abstract}

\section{Introduction}

    A graph $H$ is called a \textsl{minor} of $G$, $G\succ H$, if $H$ can be obtained from $G$ by a series of edge contractions, and vertex and edge deletions. An equivalent condition is the existence of a collection $(V_h)_{h\in H}$ of disjoint non-empty connected subsets of $V(G)$, such that for each edge $hh'\in E(H)$, there is an edge of $G$ between a vertex in $V_h$ and a vertex in $V_{h'}$ (and we say $V_h$ is adjacent to $V_{h'}$). Such a collection of subsets is called a \textsl{model} of $H$ in $G$.
    
    We define $c(H) = \inf \{c: e(G)/|G|\geq c \implies G\succ H\}$, and call this the \textit{minor extremal function} for $H$. Mader \cite{mader67} showed that this function exists for all graphs

        Previous work has studied the minor extremal function for a range of graphs --- in particular complete graphs have been considered by Mader \cite{mader67}, Kostochka \cite{KostochkaKt} and Thomason \cite{thomason_1984, ThomasonComplete}, and complete bipartite graphs with part sizes in fixed ratio $\beta t : (1-\beta)t$  by Myers and Thomason \cite{MyersThomason} --- in these cases the minor extremal function is asymptotically exactly known. For complete bipartite graphs the bound is smaller than for the same size complete graphs by a multiplicative factor $2\sqrt{\beta(1-\beta)}$ arising from treating the classes of the graph differently.
        
        For general graphs, the extremal function is known for almost all graphs with large fixed average degree and any given number of vertices, with a lower bound due to Norin, Reed, Thomason and Wood \cite{NRTW} being matched by an upper bound by Thomason and the author \cite{ThomasonWales}. In this paper, we generalise to two new cases. First, we generalise the fixed average degree result to gain a factor $\gamma_r$, a restriction of the parameter $\gamma$ of Myers and Thomason \cite{MyersThomason}, which is described in the following section.
        
        We also consider complete bipartite graphs $K_{s,t}$. K\"uhn and Osthus \cite{KuhnOsthus} (and later Kostochka and Prince \cite{KP2}) showed that for $s/t$ sufficiently small, the extremal function satisfies  $c(K_{s, t}) = (\frac12+o(1))t$  --- this is essentially tight, since a clique with $s+t-1$ vertices cannot contain such a minor.
        
        They also remarked that the $(\frac12+o(1))t$ lower bound fails to be tight once $s(\log t) /t$ is sufficiently large. This motivates a choice of scaling $s = f(t)t/\log t$, since there is a transition in behaviour of $c(H)$ from linear in $t$ for $f(t) = o(1)$ to no longer linear once $f(t)$ becomes large.
        
       In this paper, we show an asymptotically (in $f(t)\rightarrow\infty$) tight result for complete bipartite graphs, as well as for books\footnote{The book $K^*_{s,t}$ is the graph obtained from $K_{s,t}$ by adding all edges inside the class of size $s$} --- showing both $c(K^*_{ft/\log t, t})$ and  $c(K_{ft/\log t, t})$ are equal to $(2\alpha+o(1))t\sqrt{f}$, where $\alpha$ is the constant of the following definition. 
       
       \begin{dfn}
       \label{D:gammadef}
       The constant $\alpha$ is the maximum value (over $0<p<1$) of $p/(2\sqrt{\log(1/1-p)})$. We can approximate $\alpha = 0.319\dotsc$, and this value is attained at $p = 0.715\dotsc$. 
       \end{dfn}
       
       Myers and Thomason \cite{MyersThomason} introduced a parameter $\gamma(H)$ to derive bounds on $c(H)$ for non-complete graphs. They showed that, provided $\gamma(H)$ is bounded away from zero for a family of interest, this parameter determines $c(H)$ asymptotically. We state their result now in our notation --- we defer defining the parameter to the next section.

        \begin{theorem}[{\cite[Theorem~2.2]{MyersThomason}}]
        Let $H$ be a graph with $t$ vertices. Then \begin{align*}c(H) = \alpha(\gamma(H)+o_t(1))t\sqrt{\log t}.\end{align*}
        \end{theorem}
        
        This result does not, however, give a qualitative result for a sequence of graphs with $\gamma(H) = o(1)$, beyond that $c(H) = o(t\sqrt{\log t})$. While computing $\gamma(H)$ is in general hard, we can provide the following bounds. If $H$ has $td$ edges, we can bound $\gamma(H)\leq \sqrt{\frac{\log d}{\log t}}$, and if we have a complete bipartite graph $K_{\beta t, (1-\beta)t}$ then $\gamma(H) = 2\sqrt{\beta(1-\beta)} + o(1)$. Motivated by these we identify two cases of interest with $\gamma(H) = o(1)$: the first being graphs with fixed (or slow-growing) average degree, and the second case being very unbalanced complete bipartite graphs, i.e. where the left hand class\footnote{In a complete bipartite graph $K_{s,t}$, the size $s$ class will be denoted the left-hand class, and the size $t$ the right hand. Similarly, for a bipartite graph with bipartition $(A,B)$ explicitly given, we denote $A$ as the left-hand class.} has size $o(t)$. 
        
        The first of these cases has been considered by several authors. Reed and Wood \cite{ReedWood, rwcorrig} obtained an upper bound $c(H)\leq 1.9475t\sqrt{\log d}$ for all graphs $H$ with $t$ vertices and average degree $d$ sufficiently large. The following lower bound, true for almost all such graphs, was recently proven by Norin, Reed, Thomason and Wood \cite{NRTW}. They claim this result only for integer $d$, though this is not a necessary limitation of their proof. 
        
    \begin{theorem}[{\cite[Theorem 4]{NRTW}}] 
    \label{T:NRTW}
    Let $\epsilon > 0$. Then there is a $D = D_{\ref{T:NRTW}}(\epsilon)$ such that for all $t>d>D$, and all but at most $2^{-t}$ proportion of graphs $H$ with $t$ vertices and $td$ edges, there is a graph $G$ with $e(G)/|G|\geq (1-\epsilon)\alpha t\sqrt{\log d}$ which does not contain $H$ as a minor.
        
    \end{theorem}

        They asked if their result could be matched by an upper bound, which was answered affirmatively by Thomason and the author \cite{ThomasonWales}.

        \begin{theorem}[{\cite[Theorem 1.3]{ThomasonWales}}]
        Let $\epsilon > 0$. Then there is a $D>0$ such that for all $t>d>D$, and all graphs $H$ with $t$ vertices and $td$ edges, $c(H)\leq (\alpha + \epsilon)t\sqrt{\log d}$.
        \end{theorem}

        The above results can only be tight for graphs where $\gamma(H) = \sqrt{\frac{\log d}{\log t}} + o(1)$. Myers and Thomason \cite{MyersThomason} showed that this holds for almost all graphs, but that it fails for some interesting classes of graphs such as bipartite graphs with non-equal part sizes.
        
        For complete bipartite graphs, less is known. K\"uhn and Osthus \cite{KuhnOsthus} showed essentially the following result. A version with stronger conditions on $f(t)$ was independently proven by Kostochka and Prince \cite{KP1}, with a more exact bound on the extremal function under certain conditions.
        \begin{theorem}[{\cite[Theorem 2]{KuhnOsthus}}]
          For any function $f(t)$, if $H = K_{\floor{f(t)t/\log t}, t}$ then \\$c(H) = (\frac12+o_{1/f,t}(1))t$ (i.e. for all $\epsilon > 0$, there is a constant $\delta$ such that for all $t>1/\delta$ and any $f <\delta$, $c(H)\leq (\frac12+\epsilon)t$). 
        \end{theorem}
        Kostochka and Prince \cite{KP2} showed the above result remains true replacing $K_{s,t}$ with  $K^*_{s,t}$. K\"uhn and Osthus \cite{KuhnOsthus} also gave the following lower bound, which shows in particular that no bound of the form $c(K_{ft/\log t, t})\leq Ct$ can hold for all $t$ if $f$ exceeds a (large) constant.
        
        \begin{theorem}[{\cite[Proposition 9]{KuhnOsthus}}]
            For each $\beta > 0$, there is an $n_0$ such that for all $n > n_0$, there is a graph with average degree at least $n/2$, and no $K_{s,t}$ minor, where $t = \ceil{\beta n}$ and $s = \ceil{\frac{2n}{\beta \log n}}$. In particular, we have $c(K_{ft/\log t, t})\geq t\sqrt{f/32} \sim 0.177 t\sqrt{f}$, for $t$ sufficiently large depending on $f$. 
        \end{theorem}
        \begin{remark}
        In fact, re-analysing their proof and using a more optimal choice of parameters shows a lower bound $(\alpha - o(1))t\sqrt{f(t)}\sim 0.319t\sqrt{f}$ holds.
        \end{remark}

        Having completed these preliminaries, we can now start to state our results.  We start with a pair of asymptotically matching bounds for complete bipartite graphs. The upper bound can be viewed as a special case of the later Theorem~\ref{T:mainweightthm}, though we give a separate and more explicit proof.
    \begin{restatable}{theorem}{mainthmbip}
    \label{T:mainthmbip}
        Let $\epsilon >0$. Then there is a constant $C = C_{\ref{T:mainthmbip}}(\epsilon)$ such that for all $t,f(t)$ satisfying $\log t > f(t) > C$, any graph $G$ satisfying $e(G)/|G|\geq(2\alpha+\epsilon)t\sqrt{f}$ contains $K_{ft/\log t, t}$ as a minor, where $\alpha = 0.319\dotsc$ is the constant from Definition~\ref{D:gammadef}. In particular, $c(K_{ft/\log t, t})\leq (2\alpha +\epsilon)t\sqrt{f}$.
    \end{restatable}
        
\begin{restatable}{theorem}{lowerboundbipartite}
\label{T:lowerboundbipartite}
Let $\epsilon>0$. Then there is a constant $T = T_{\ref{T:lowerboundbipartite}}(\epsilon)$ such that for all $T< f(t) \leq \log t$, there is a graph satisfying $e(G)/|G|\geq(2\alpha-\epsilon)t\sqrt{f}$ and no $K_{ft/\log t,t}$ minor. In particular, $c(K_{ft/\log t, t})\geq (2\alpha - \epsilon)t\sqrt{f}$.
\end{restatable}

    We now move to our theorems for structured sparse graphs. The following definition was motivated by the parameter $\gamma$ (and links between these are explored in the next section). This is then used to define the class of graphs for which we have matching bounds on $c(H)$. We let $\mathbb{R}_{\geq 0}$ denote the set of non-negative real numbers, and let $[-\infty,1]$ denote all real numbers at most 1, together with a formal symbol $-\infty$ "smaller than all real numbers".
    
         \begin{dfn}
   \label{D:vectorgamma}
    Suppose that $\parts = (\sigma_1,\dotsc,\sigma_r)\in \mathbb{R}_{\geq 0}^r$ is a weight vector (i.e. $\sum \sigma_i = 1$), and $\admat = (D_{ij})_{i=1}^r$ is a symmetric matrix, where $D_{ij}\in [-\infty,1]$. Define $\gamma(\parts,\admat) = \min_{\beta\in \mathbb{R}_{\geq 0}^r} \{\parts\cdot\beta: \beta_i\beta_j \geq D_{ij}\textrm{ for } 1\leq i,j\leq r\}$, where $\cdot$ denotes the usual inner product.
    \end{dfn}
        
        \begin{dfn}
\label{D:vectorclass}
    Let $\parts = (\sigma_1,\dotsc,\sigma_r)$ and $\admat$ be as in Definition~\ref{D:vectorgamma}. Define the class of graphs  $\mathcal{D}(\parts,\admat)(t,d)$ to be all graphs $G$ with $V(G) = \{1,\dotsc,t\}$
       and the following property.
       
       Let $s_i = \floor{\sigma_i t}$ for $i\leq i_0$, and $\ceil{\sigma_i t}$ for $i>i_0$, where $i_0$ is chosen so that $\sum s_i = t$.
       Let $S_1 = \{1,\dotsc,s_1\}$, $S_2 = \{s_1+1,\dotsc,s_1+s_2\}$ and so on. Then $G$ has $\floor{td^{D_{ij}}}$ edges between $S_i$ and $S_j$. 
       
       Let $\mathcal{D}(\parts,\admat)$ be the union of all classes $\mathcal{D}(\parts,\admat)(t,d)$.
    \end{dfn}
        
        The formal symbol $-\infty$, with the property $d^{-\infty} = 0$ for any $d>1$, is permitted to allow no edges between some pairs in the definition of $\mathcal{D}(\parts,\admat)$. We note that since $\beta$ is a non-negative vector, we can increase any negative $D_{ij}$ to $0$ without affecting the value of $\gamma(\parts,\admat)$.
        
        One might be tempted to avoid this technicality by relaxing the definition of the graph class to allow at most $td^{D_{ij}}$ edges between $S_i$ and $S_j$. However, our lower bounds are of the form `almost all', and so when we allow both graphs with $t$ edges and graphs with no edges, the latter are entirely `hidden from view' by the former, since there are many more of them.  Our restrictive definition means that for any choice of number of edges, we can pick a corresponding $\parts$ and $\admat$ (for each $t,d$), and therefore our theorem is as applicable as possible.
        
        We were able to prove the following matching upper and lower bounds for this setting; the upper bound being a slightly special case and reformulation of a general Theorem~\ref{T:mainweightthm}.
        
        \begin{theorem}
        \label{T:gammaub}
        Let $\epsilon>0$, and $r\geq 1$ an integer. Then there is a constant $D = D_{\ref{T:gammaub}}(\epsilon,r)$ such that the following holds. Let $\parts$ be a weight vector of length $r$, and $\admat$ a matrix with entries in $[-\infty,1]$ satisfying $\gamma(\parts,\admat)>\epsilon$. For all $t>d>D$, every element $H\in \mathcal{D}(\parts,\admat)(t,d)$ has $c(H)\leq (\alpha+\epsilon)\gamma(\parts,\admat)t\sqrt{\log d}$.
        \end{theorem}

            \begin{restatable}{theorem}{gammalb}
    \label{T:extremalfunction}
        Let $\epsilon > 0$. Then there is a $D = D_{\ref{T:extremalfunction}}(\epsilon)$ such that the following holds for all $t>d>D$.
        
    Let $\parts$ be a weight vector with all $\sigma_i > \epsilon$, and $\admat$ a matrix with entries in $[-\infty,1]$ such that $\gamma(\parts,\admat)\geq \epsilon$. Then all but at most $2^{-t}$ proportion of elements $H\in \mathcal{D}(\parts,\admat)(t,d)$ have $c(H)\geq (\alpha-\epsilon) \gamma(\parts,\admat) t\sqrt{\log d}$, where $\alpha = 0.319\dotsc$ is the constant from Definition~\ref{D:gammadef}.
    \end{restatable}

    \section{The parameters $\gamma$ and $\gamma_r$}
    
    In \cite{MyersThomason}, Myers and Thomason introduced the parameter $\gamma(H)$, defined as follows.
    
    \begin{dfn}
    \label{D:gamma}
    Let $|H| = t$. The parameter $\gamma(H)$ is the minimum of $\frac{1}{t}\sum_v w(v)$ over all functions $w: V(H)\rightarrow \mathbb{R}_{\geq 0}$ satisfying $\sum_{E(G)}t^{-w(u)w(v)}\leq t$.
    \end{dfn}
    
     Unfortunately, while the definition of $\gamma$ was helpful for showing extremal graphs are pseudorandom, it is in general quite difficult to calculate exactly. One weight function which satisfies the inequality of Definition~\ref{D:gamma} when $H$ has $td$ edges is $w(v) \equiv \sqrt{\frac{\log d}{\log t}}$, and so this quantity is an upper bound on $\gamma(H)$. This upper bound is tight for almost all graphs, see \cite{MyersThomason}. One can view $\gamma(\parts,\admat)$ as a restriction of this parameter. It can also be thought of as a generalisation of the notion of `shapes' from that paper to allow for weighted edges.  Our aim in this section is to convert the bound from Theorem~\ref{T:gammaub} to a bound in terms of an explicit graph parameter similar to $\gamma(H)$.

    \begin{remark}\label{R:degreevstwice}It is sometimes taken that $H$ has average degree $d$, instead of having $td$ edges (and hence average degree $2d$). This does not make a significant difference to any of our results, since $d$ only appears inside logarithms in bounds, and \\$\log 2d \leq 1 + \log d$  - i.e. this only affects bounds by a multiplicative $(1+o(1))$ factor which is easily absorbed in error terms.\end{remark}
    
    Due to technical issues requiring rounding, the methods of Myers and Thomason do not apply in the case $\gamma(H) = o(1)$, and in particular to the case $\log d = o(\log t)$. 
    The following definition restricts to a fixed number of different weights to avoid this issue. We also incorporate the $\sqrt{\frac{\log d}{\log t}}$ factor into our new definition; our parameter can then attain values between 0 and 1 for all $t,d$. 
    
    \begin{dfn}
    A \textsl{weighted partition} of $H$ (into $r$ parts) is a collection $(P_i,w_i)_{i=1}^r$ where the $P_i$ form a partition of $V(H)$, and the $w_i$ are non-negative real numbers (called weights).
    
    Suppose that $|H| = t$ and $e(H) = td$. Given a weighted partition of $H$, for $1\leq i,j \leq n$ define quantities $D_{ij}\in [-\infty,1]$ by $e_H(P_i,P_j) = td^{D_{ij}}$. Such a weighted partition satisfies the \textsl{Gamma inequality} if the following holds.
    \begin{align}
            \label{eq:gammainequalitya}\sum_{i\leq j} d^{D_{ij}-w_iw_j}\leq 1
    \end{align}
    Define $\gamma_r(H)$ to be the minimum of $\frac1t\sum_i w_i |P_i|$ over all weighted partitions of $H$ into $r$ parts satisfying the Gamma inequality (\ref{eq:gammainequalitya}).
    \end{dfn}

    
    \begin{remark}
    \label{R:gammaineqconstant}
        The constant on the right-hand side of inequality \ref{eq:gammainequalitya} is effectively arbitrary. Suppose that a weighting $w$ satisfies the inequality, but with 1 replaced with a constant $C>1$. Then increasing each $w_i$ to $(1+\sqrt{\log_d C})w_i$, this new weighting satisfies inequality (\ref{eq:gammainequalitya}) as stated above, and so as in Remark~\ref{R:degreevstwice} $\gamma_r$ only changes by a $(1+o_d(1))$ factor. The choice 1 was made for simplicity and to naturally relate to $\gamma(H)$.
    \end{remark}
    
    For a weighted partition to satisfy the Gamma inequality (\ref{eq:gammainequality}), it is necessary that $w_iw_j \geq D_{ij}$ for all $i,j$. Further, if we have a weighting satisfying all of these inequalities, it satisfies a version of the Gamma inequality with the quantity 1 on the right hand side replaced with $\binom{r+1}{2}$. For sufficiently slow-growing $r = d^{o(1)}$, this is equivalent by Remark~\ref{R:gammaineqconstant}.

    Having defined and somewhat explained this parameter, we are now able to state our upper bound in terms of it.  In the following subsection, we will explain how to derive Theorem~\ref{T:gammaub} from this.

    \begin{restatable}{theorem}{mainthmgamma}
    \label{T:mainthmgamma}
    Let $\epsilon,r > 0$. Then there is a constant $C = C_{\ref{T:mainthmgamma}}(\epsilon,r)$ such that if $H$ is a graph with $t$ vertices, average degree $d>C$ and $\gamma_r(H)\sqrt{\log d} \geq C$, then any graph $G$ with $e(G)/|G|$ at least $(\alpha+\epsilon)\gamma_r(H)t\sqrt{\log d}$ contains $H$ as a minor. 
    \end{restatable}
    
    This extends prior results of Myers and Thomason (who proved a slight generalisation of this theorem provided that $d$ grows sufficiently quickly, and $\gamma_r(H)$ remains bounded away from zero). The above theorem is essentially best possible --- it is asymptotically tight for some notion of `almost all graphs'.

\subsection{Lower bounds, and the parameter $\gamma(\parts,\admat)$}

While $\gamma_r(H)$ turns out to be helpful for proving upper bounds, the varied structures of graphs with fixed $\gamma_r$ prove problematic for proving lower bounds. The parameter $\gamma(\parts,\admat)$ uses the structure of graphs in $\mathcal{D}(\parts,\admat)$ and so is easier to prove lower bounds with. The definition of $\gamma(\parts,\admat)$ is broadly similar to that of $\gamma_r$, except that we restrict each quantity $td^{D_{ij}-w_iw_j}$ to be at most 1, rather than their sum. This is, however, only a slightly more restrictive version of satisfying the Gamma inequality with the constant 1 replaced with $\binom{r+1}{2}$. Bounding the contribution from each pair is often easier than bounding their sum, and for sufficiently small $r$ this is equivalent as noted after Remark~\ref{R:gammaineqconstant}. This equivalence only holds for sufficiently slow-growing $r$; we will consider only the regime where $r$ is constant and $d$ is large.
   
    
    From the above comments, if $H$ is a member of $\mathcal{D}(\parts,\admat)(t.d)$, and further $\max D_{ij} = (1+o(1))$ (this is required so that $d$ is the base in both definitions), we have $\gamma_r(H) \leq \gamma(\parts,\admat)(1+o(1))$ by considering this choice of partition, and the slightly increased weighting. This shows that Theorem~\ref{T:gammaub} follows from Theorem~\ref{T:mainthmgamma}, since we have provided a weighting satisfying the Gamma inequality, and hence bounded $\gamma(H)$ above.
   
    Another application of the class $\mathcal{D}(\parts,\admat)$ is that in order to prove a tight lower bound on the minor extremal function, we need some notion of `almost all' graphs for it to apply to. The following theorem shows Theorems~\ref{T:gammaub} and ~\ref{T:extremalfunction} do form a matching pair of bounds --- $\gamma(\parts,\admat)(t,d)$ consists almost entirely of graphs for which the upper and lower bound match.

\begin{restatable}{theorem}{gammafamilymatch}
\label{T:gammafamilymatch}
Let $\gamma,\epsilon,r > 0$. Then there is a $D = D_{\ref{T:gammafamilymatch}}(\gamma,\epsilon,r)$ such that for all $t>d>D$, all $\parts$ weight vectors of length $r$ and $\admat$ matrices satisfying the conditions of Definition~\ref{D:vectorgamma} such that $\gamma(\parts,\admat) \geq \gamma$; for all but at most $\epsilon$ proportion of graphs $H$ from $\mathcal{D}(\parts,\admat)(t,d)$ the following holds. 
\begin{align*}
    (1-\epsilon)\gamma(\parts,\admat)\leq \gamma(H)\sqrt{\log t / \log d}\leq \gamma_r(H)\leq \gamma(\parts,\admat)(1+\epsilon)
\end{align*}

\end{restatable}

\subsection*{Example}

    We now given an example where our framework can be applied. Let $H$ be a graph with vertex partition $A\cup B\cup C$, where $|A| = |B| = |C| = t/3$. We can bound $\gamma(H)$ by consideration of $\gamma_3(H)$, and in turn bound $\gamma_3(H)$ by considering only this partition.
    
    If there were $\Theta(t^2)$ edges between all pairs, by consideration of the Gamma inequality we would be forced to take a weight of essentially 1 on all vertices, and hence make no gain over existing bounds. However, suppose instead that there are only $t^{3/2}$ edges between $B$ and $C$, with $\Theta(t^2)$ between $A$ and each of $B$ and $C$. In this regime, we must have $d = \Theta(t)$, and since all of our bounds only use $\log d$ for simplicity we can just treat $d = t$.
    
    When considering this partition, we can view such an $H$ as an element of $\gamma(\parts,\admat)(t,d)$, with $\parts = (\frac13,\frac13,\frac13)$, and $D_{ii} = -\infty$, $D_{12} = D_{13} = 1$ and $D_{23} = \frac12$ (this is the limiting case). We remark it does not affect our bounds to allow up to $t$ edges within each class --- this corresponds to increasing $D_{ii}$, but they would remain negative.
    
    Suppose we assign weights $w_1,w_2,w_3$ to $A,B,C$ respectively. In order for these to provide a bound on $\gamma(\parts,\admat)$, we require $w_1w_2\geq 1$, $w_1w_3\geq 1$, and $w_2w_3\geq \frac12$ --- we would like to minimise $\sum w_i$. By symmetry, it is optimal to consider $w_2 = w_3$, and it is always optimal to take $w_2 = 1/w_1$. Writing $x$ for the value $w_2 = 1/w_1$, it remains to minimise $2x + 1/x$ over all $x\geq \frac{1}{\sqrt{2}}$.  It turns out that taking $x = \frac{1}{\sqrt{2}}$ is optimal, and so we can compute $\gamma(\parts,\admat)= 2\sqrt{2}/3 < 1$. By Theorem~\ref{T:gammafamilymatch}, for almost all graphs constructed in this fashion $\gamma_3(H) = (1+o(1))2\sqrt{2}/3$ holds.
    
    Since $2x+1/x$ has a global minimum at $x = 1/\sqrt{2}$, if we instead took $D_{23} = y<\frac12$ it remains optimal to take $w_2 = 1/\sqrt{2}$. If we take $d_{23} = y\geq\frac12$, it is optimal to take $w_2 = 1/\sqrt{y}$. We remark that we could have attained this bound using $\gamma_2$ by merging the classes $B$ and $C$ into a single class of size $2t/3$ with $O(t^{\frac32})$ edges inside.

\section{Breaking down the upper bounds}

In this section, we will break our existing upper bound Theorems~\ref{T:mainthmbip} and \ref{T:mainthmgamma} into separate, smaller theorems for dense and sparse cases. We will also state some properties of minor-minimality, and show how these together imply the aforementioned theorems --- in fact, we will prove the strengthed Theorem~\ref{T:mainweightthm} instead of Theorem~\ref{T:mainthmgamma}. To complete the proofs of our upper bounds, it will then only remain to prove Theorems \ref{T:ctdbipartite}, \ref{T:ctdgamma}, and \ref{T:sparsetheorem}.

\begin{definition}
    A graph $G$ is minor-minimal in a class $\mathcal{C}$ of graphs if $G\in \mathcal{C}$, but no proper minor of $G$ is a member of $\mathcal{C}$. 
\end{definition}

    As in \cite{ThomasonComplete}, we introduce a class of graphs \emk{} that in particular contains all graphs with $e(G)/|G|\geq m$. It will then suffice to prove our results only for minor-minimal elements of this class. This requires us to sacrifice a small amount of average degree, but in exchange we gain some useful properties.
    
 \begin{definition}
 Let $2m>k>1$. We define \emk{} to be the class of all graphs $G$ with $|G|\geq m$ and $e(G)\geq m|G|-km$.
 \end{definition}

\begin{lemma}[{\cite[Lemma 1.5]{ThomasonWales}}]
  \label{L:enkprop}
  Let $G$ be a minor-minimal element of $\mathcal{E}_{m,k}$. Then
  $|G|\geq m+1$, $e(G)\leq m|G|-mk+1$,
  $m<\delta(G)<2m$, $\kappa(G)>k$, and every edge of $G$ is
  in more than $m-1$ triangles.
\end{lemma}

The proof is elementary, see \cite{ThomasonComplete,ThomasonWales} for a proof.

\subsection{The dense case}

In the dense case, we will have $|G|\leq Dm$ for some constant $D$, and also $\kappa(G)\geq \eta |G|$ for some constant $\eta$. In this case with `positive fraction connectivity', we are able to build up a random partition labelled by vertices of $H$ and use the connectivity to turn this into a minor. In fact it will turn out to be useful to have the following, stronger, condition of being able to place certain vertices into the minor.

\begin{definition}
       A graph $G$ is called $H$ minor \textit{prevalent} if $|G|\geq |H|$, and for every subset $R = \{r_h: h\in H\}\subset V(G)$ of $|H|$ distinct vertices (called a set of \textsl{roots}), $G$ has a $H$ model $(V_h)$ such that $r_h\in V_h$ - in other words, $G$ has a $H$ minor at any choice of roots.
\end{definition}
 
\begin{restatable}{theorem}{ctdtheorem}
\label{T:ctdbipartite}
Let $0<\eta,\epsilon,p<1$. Then there is a constant $C_{\ref{T:ctdbipartite}}(\epsilon,\eta) = C$ such that if $G$ is a graph with $n\geq 2t\sqrt{f/\log(1/(1-p))}$ vertices, density at least $p+\epsilon$, and connectivity at least $8\eta |G|$, then $G$ is $K^*_{ft/\log t,t}$ minor prevalent, provided that $C<f<\log t$.
    \end{restatable}

\begin{restatable}{theorem}{gammactdtheorem}

\label{T:ctdgamma}
Let $\epsilon,r,\eta>0$. There is a constant $T = T_{\ref{T:ctdgamma}}(\epsilon,r,\eta)$ such that the following holds for any $t>T$.
Let $H$ be a graph with $t$ vertices, and average degree~$d$. Let $(P_i,w_i)_{i=1}^r$ be a weighted partition of $H$ satisfying the Gamma inequality, restated as follows.
    \begin{align}\label{eq:gammainequalityb}\sum_{i\leq j} d^{-w_iw_j}e(P_i,P_j)/t \leq 1.\end{align}
    
    Let $w = \sum_i |P_i|w_i$, and suppose $w \geq Tt /\sqrt{\log_{1/1-p}d}$. Let $G$ have density at least $p+\epsilon$, $n\geq w \sqrt{\log_{1/1-p}d}$ vertices and connectivity at least $\eta |G|$. Then $G$ is $H$ minor prevalent
\end{restatable}

One might be tempted to instead state the above theorem in terms of $\gamma_r$, rather than a general weighting. However,  this formulation means we do not need to evaluate $\gamma_r$ to ensure we can apply the theorem, and that additional generality is helpful for our proofs. This motivates also generalising Theorem~\ref{T:mainthmgamma} as below --- picking an optimal weighting gives the earlier Theorem~\ref{T:mainthmgamma} as an immediate corollary.

\begin{theorem}
\label{T:mainweightthm}
Let $\epsilon,r > 0$. Then there is a constant $C = C_{\ref{T:mainweightthm}}(\epsilon,r)$ such that the following holds.

Let $H$ be a graph with $t$ vertices and average degree $d>C$, equipped with a weighted partition $(P_i,w_i)_{i=1}^r$ satisfying the Gamma inequality (\ref{eq:gammainequalityb}).

Then if $w = \sum_i |P_i|w_i$ satisfies $ w \geq Ct/\sqrt{\log d}$, then any graph $G$ with $e(G)/|G|\geq(\alpha + \epsilon)w\sqrt{\log d}$ contains $H$ as a minor.
\end{theorem}

\subsection{The sparse case}

In the sparse case $|G| >>m$, we are able to use the size of $G$ to build up many disjoint small dense subgraphs, and find different parts of $H$ as minors in these different subgraphs. We can then use connectivity to join these minors together.

  \begin{restatable}{theorem}{sparsetheorem}
  \label{T:sparsetheorem}
    Let $k>0$ be an integer, $0<\epsilon<1/500$. Then there is a constant $D = D_{\ref{T:sparsetheorem}}(\epsilon,k)$ such that the following holds.
    Let $H$ be a graph, and $H_1,\dotsc,H_k$ graphs such that $\bigcup_i H_i = H$.
    
    Suppose that $m>D|H|$ is such that for each $1\leq i \leq k$, every graph $\widetilde{G}$ with minimum degree at least $\epsilon m$ and connectivity at least $\epsilon |\widetilde{G}|$ is $H_i$ minor prevalent. Suppose further that every graph  $\gbar$ with at least $m/7$ vertices, and minimum degree at least $(1-\epsilon)|\gbar|$ contains $H$ as a minor. 
    
    Then every graph $G$ with properties (1)-(4) below contains $H$ as a minor.
    
    \begin{enumerate}
        \item $e(G)\leq m|G|$
        \item Every edge of $G$ is in at least $m-1$ triangles
        \item $\kappa(G)\geq D|H|$
        \item $|G|\geq Dm$
    \end{enumerate}
    \end{restatable}

\subsection{Proof of Theorems~\ref{T:mainthmbip} and \ref{T:mainweightthm} }

\begin{proof}[Proof of Theorem~\ref{T:mainthmbip}]
    Reduce $\epsilon$ if necessary so that $\sqrt{\log(1/2\epsilon)} > 28/\alpha$; note this only strengthens the result. Let $H = K_{ft/\log t,t}$, $m = (2\alpha + \epsilon)t\sqrt{f}$, and $k = \epsilon m$. Let $G$ be a graph as in the statement, and note that since $e(G)/|G|\geq m$, $G\in \emk$. Replacing $G$ by a minor if necessary, we can assume $G$ is minor-minimal in \emk{} --- recall this implies $G$ has the properties of Lemma ~\ref{L:enkprop}. 
    
    Let $N$ be an integer yet to be determined. Partition each vertex class of $H$ into $N$ almost equal size parts, and add additional vertices of $H$ to these parts (no longer requiring that parts be disjoint) so that the parts of the left hand class have size exactly $\ceil{ft/N\log t}$, and $\ceil{t/N}$ for parts of the right hand class. By taking the induced subgraph on pairs of parts, we get an edge-cover of $H$ by $N^2$ subgraphs $(H_{i})$, each isomorphic to $K_{\ceil{ft/N\log t}, \ceil{t/N}}$. 
    
    We would like to apply Theorem~\ref{T:sparsetheorem} to $H$ and the partition $(H_i)$ to $N^2$ parts. For this, we need to show graphs $\gtilde$ with minimum degree at least $2\epsilon\alpha t\sqrt{f}$ and connectivity at least $\epsilon |\widetilde{G}|$ are $H_i$ minor prevalent. 
    
    By the connectivity condition, $\gtilde$ has minimum degree at least $\epsilon |\gtilde|$ and hence density $p \geq \epsilon$. The definition of $\alpha$ implies the following useful inequality for any $0<p<1$, which we immediately apply to bound $e(\gtilde)/|\gtilde|$.
    \begin{align}\label{eq:alphaproperty}
        \frac{\alpha}{p}\geq \frac{1}{2\sqrt{\log(1/(1-p))}}
        \end{align}

    \begin{align*}\frac{p}{|\widetilde{G}|}\binom{|\gtilde|}{2} = e(\gtilde)/|\gtilde|\geq \epsilon \alpha t\sqrt{f} \geq \epsilon t\sqrt{f} \frac{p}{2\sqrt{\log(1/(1-p))}}\end{align*} Therefore, $|\gtilde| \geq \epsilon t\sqrt{f/\log(1/(1-p))}$. Provided we ensure $N > 8/\epsilon$ (and we now fix such an $N$), we have $|\gtilde| \geq \ceil{\frac{t}{N}} \sqrt{\frac{\ceil{ft/N\log t}}{\ceil{t/N}\log(1/(1-p))}}$ for any $f>4N$, so we can apply Theorem~\ref{T:ctdbipartite} to deduce $\gtilde$ is $H_i$ minor prevalent; provided that $|H_i| \geq C/N^2$ is sufficiently large.
    
    We also need to find $H$ directly as a minor in a very dense graph. Let $\gbar$ be a graph with $|\gbar|\geq \alpha t\sqrt{f}/7 \geq 2t\sqrt{f/\log(1/2\epsilon)}$ and minimum degree at least $(1-\epsilon)|\gbar|$, and suppose it has density $p$. In particular, \gbar{} is $|\gbar|/3$ connected since $\epsilon < 1/6$. Therefore, \gbar{} satisfies the hypotheses of Theorem~\ref{T:ctdbipartite} applied with $H$, the current $\epsilon$, $\eta = 1/48$ and $p$ taken as $p-\epsilon$ (this is non-negative, since $G$ has minimum degree at least $|G|/3$). This imposes some lower bound $C > C_{\ref{T:ctdbipartite}}(\epsilon,1/48)$.
    
    In particular, we satisfy the conditions of Theorem~\ref{T:sparsetheorem} (recall that $G$ satisfies properties (1)-(3) by Lemma~\ref{L:enkprop}), applied with $k = N^2$ and the above choice of $(H_i)$ provided $C$ is sufficiently large. There is therefore a constant $D = D_{\ref{T:sparsetheorem}}(\epsilon) (>1)$ such that if $|G|\geq Dm$, and $m>D|H|$ then $G$ contains $H$ as a minor. We note that $m > 0.5t\sqrt{f(t)} >0.5t\sqrt{C}$ and therefore provided $C > 4D^2$ this is satisfied.
    
    The above argument shows that $H$ is a minor of $G$ if $|G|\geq Dm$, and so from here we may assume $|G|\leq Dm$. In this `dense' case, $G$ has connectivity at least $(\epsilon/D) |G|$ hence also minimum degree $(\epsilon/D)|G|$. Suppose that $G$ has density $p$, and let $p' = p - \epsilon/4D > \epsilon/4D$. By inequality~(\ref{eq:alphaproperty}), the following holds. \begin{align*}e(G)/|G| \geq 2(1+\epsilon)p\alpha t\frac{\sqrt{f}}{p} \geq(1+\epsilon)pt\frac{\sqrt{f}}{\sqrt{\log(1/(1-p))}}\end{align*} Hence $|G| \geq 2(1+\epsilon)t\sqrt{f/\log(1/(1-p))}\geq 2t\sqrt{f/\log(1/(1-p'))}$ as above, and so provided that $C$ is sufficiently large (depending on all previous constants) $G$ is $H$ minor prevalent, and in particular contains $H$ as a minor. 
\end{proof}

\begin{proof}[Proof of Theorem~\ref{T:mainweightthm}]
      Reduce $\epsilon$ if necessary so that $\epsilon < 1/500$. We remark that $G$ is an element of the family $\mathcal{E}_{m,k}$ for $m=(\alpha + \epsilon)w\sqrt{\log d}$ and $k = \epsilon m$. Replacing $G$ by a minor if necessary, we assume $G$ is minor minimal in $\mathcal{E}_{m,k}$.
      
      We first handle the case where $|G|$ is very large. Let $N$ be a (large) integer yet to be determined, and we will later pick $C$ depending on $N$. We construct a partition of the vertex set $V(H)$ as follows. Let $\func{w}{V(H)}{\mathbb{R}_{\geq 0}}$ be the weight function mapping each $v\in P_i$ to $w_i$. This can be extended additively to a function on $\mathcal{P}(V(H))$, with $w(V(H)) = w$.
      
      Let $B$ consist of all vertices of weight at least $2w/N$, and clearly $|B|\leq N/2$. Reordering if necessary, let $P_1$ be a part of maximal size, so in particular \\$|P_1|\geq|H|/r$. We will take $C>Nr$ so that $B$ and $P_1$ are disjoint (recall that $w$ is constant on parts). We start by partitioning $H' = H[V(H)\setminus(P_1\cup B)]$. Let $V_1,\dotsc,V_N$ be $N$ initially empty sets, which we call \textit{bags}. We place the vertices of $H'$ into these bags one at a time in non-increasing degree order. When we come to add $v$, we place it arbitrarily into any $V_i$ such that $w(V_i) \leq w/N$. Such an index must exist because the total weight is at most $w$. Continue in this fashion until $\cup V_i = V(H')$.
      
      At the end of this procedure, any vertex of weight at least $w/N$ must lie in a bag on its own (consider the largest weight vertex $v$ in such a bag; no other vertices can be added after $v$ since the weight is too large and so the bag contains only one vertex). In particular, the total weight $w(V_i)$ must be at most $2w/N$. We now add $P_1$ to these bags. We place either $\floor{|P_1|/N}$ or $\ceil{|P_1|/N}$ vertices from $P_1$ into each $V_i$, so that they remain disjoint, and $\cup V_i = V(H)\setminus B$. Since $|P_1|/N \geq 4$ (for $C>4rN$), we have that $w(P_1)/2N\leq w(V_i\cap P_1)\leq 2w(P_1)/N$. Combining these bounds, each bag $V_i$ must have weight at most $4w/N$.
      
      Let $(W_i)_{i=1}^N$ be some arbitrary partition of $V(H)\setminus B$ into $N$ sets, each of size either $\floor{\frac{|V(H)\setminus B|}{N}}$ or $\ceil{\frac{|V(H)\setminus B|}{N}}$. We can now construct our edge cover of $H$, consisting of the following $k = \binom{N}{2} + N$ graphs. We aim to apply Theorem~\ref{T:sparsetheorem} with these graphs, and this value of $k$.

      \begin{itemize}
          \item Let $H_{0,i}$ be the graph on vertex set $V(B)\cup W_i$, and all edges with at least one endpoint in $B$ for $1\leq i \leq n$.
          \item Let $H_{a,b} = H[V_a\cup V_b]$ for $1\leq a <b\leq N$.
      \end{itemize}
      
      We would like to proceed using Theorem~\ref{T:ctdgamma} and a restriction of the existing weighting for the graphs $H_{a,b}$, and directly use Theorem~\ref{T:ctdbipartite} for $H_{0,j}$. However, since the gamma inequality involves the order and average degree of the graph, some additional modifications must be to the weighting. We also have to consider the case where the average degree of $H_{a,b}$ is too small to apply the desired theorem.
      
      Let $d_{a,b} = e(H_{a,b})/|H_{a,b}|$. By either a result of \cite{ThomasonWales}, or equivalently applying Theorem~\ref{T:ctdgamma} using a constant weighting, there is a constant $D_0$ such that any graph \gbar{} with $e(\gbar{})/|\gbar|\geq |H_{a,b}|\sqrt{\log d_{a,b}}$ and $\epsilon |\gbar|$ connectivity is $H_{a,b}$ minor prevalent provided $d_{a,b} \geq D_0$ for some constant $D_0$ depending only on $\epsilon$. Further, for any constant $D_1>D_0$ (which can depend on previous constants, but cannot depend on $C$), by adding edges to $H_{a,b}$ if necessary and taking $C$ sufficiently large depending on $D_1$, if $d_{a,b}\leq D_1$ then any graph with $e(\gbar)/|\gbar| \geq \epsilon m \geq t\sqrt{\log(2D_1)}$ and connectivity at least $\epsilon |\gbar|$ is $H_{a,b}$ minor prevalent. We will fix $D_1$ later, although from now we assume $d_{a,b}> 1$ --- we will only use the remainder of the proof in the case $d_{a,b} > D_1$.
      
      We define a new weighting $w_{a,b}(v) = (w(v)+\delta)\sqrt{\frac{\log d}{\log d_{a,b}}}$, where \\$\delta=\sqrt{\log(Nr) / \log d}$ is chosen so that the gamma inequality holds on $H_{a,b}$ for each $1\leq a < b \leq N$ with $d_{a,b}>1$ with the weighting $w_{a,b}$. We include the calculation below, and for simplicity let $w'_i$ denote $w_{a,b}(v)$ for some $v\in P_i$ (recall that $w$ is constant on parts, so our choice does not matter). 
      
      \begin{align*}
\sum_{i\leq j}e(P_i\cap H_{a,b}, P_j\cap{H_{a,b}})d_{a,b}^{-w'_iw'_j}/|H_{a,b}| &\leq \frac{1}{|H_{a,b}|d^{\delta^2}} \sum_{i,j} e(P_i,P_j) d^{-w_iw_j} \\&\leq \frac{t}{d^{\delta^2}|H_{a,b}|} \leq 1 \end{align*}

In particular, we have$\sqrt{\log d_{a,b}}w_{a,b}(H_{a,b}) \leq (w(H_{a,b})+\delta)\sqrt{\log d}$ and so if we require that $N$ satisfies $16\sqrt{\log Nr}/ N < \epsilon$, any graph $G'$ with average degree at least $\epsilon n$ and connectivity $\epsilon |G'|$ is $H_{a,b}$ prevalent by Theorem~\ref{T:ctdgamma} applied with $\eta$ and $\epsilon$ taking the value $\epsilon$, and the weighting $w_{a,b}$ described above --- provided that $d(H_{a,b}) > T_{\ref{T:ctdgamma}}$. Taking $D_1 = 2T_{\ref{T:ctdgamma}}$ (and note that $D_1$ does not depend on $C$, so we are justified in doing so), from one of the above arguments the result follows regardless of $d_{a,b}$.

      We next consider the graph $H_{0,j}$. Each of these is a subgraph of a book $K^*_{s,t'}$, where $s\leq N$ and $t'\leq 2|H|/N$. Applying Theorem~\ref{T:ctdbipartite}, there is a constant $C'$ such that for all $\log t' > f > C'$, any graph with connectivity at least $\epsilon t'\sqrt{f}$ and $e(G)/|G|$ at least $(2\alpha + \epsilon)t'\sqrt{f}$ is $K^*_{ft't'/\log t',t'}$ prevalent. (this result also follows from Theorem~\ref{T:ctdgamma}).

      In particular, taking $t$ sufficiently large (so that $t'$ is sufficiently large), and $N$ large, $H_{0,j}$ is a minor of any graph $G'$ with average degree at least $\epsilon n$ and connectivity $\epsilon |G'|$ (note that increasing to $s= 2C'\log(t')/t'$ strengthens the result). We can now fix our choice of $N$ --- note that we have not yet had to choose $C$.
      
      It remains only to verify that $H$ is a minor of suitable dense graphs. Let $\gbar$ be a graph with minimum degree at least $(1-\epsilon)|\gbar|$ and at least $m/7$ vertices. In particular, $\gbar$ must be $|\gbar|/3$ connected since $\epsilon < 1/6$. Suppose that the density of \gbar{} is $p$. We have, $e(\gbar)/|\gbar| \geq (1-2\epsilon)\alpha wt / 7 \geq wt/\sqrt{\log(1/(2\epsilon))}$, and so applying Theorem~\ref{T:ctdgamma} with $p$ replaced with $p-\epsilon$, our value of $\epsilon$, and $\eta = 1/3$ implies that provided $d$ is sufficiently large, $H$ is a minor of $\gbar$. 
      
      Therefore, with our value of $m$ the conditions of  Theorem~\ref{T:sparsetheorem} hold (recall that $G$ is minor-minimal and hence has the properties (1)-(3) by Lemma~\ref{L:enkprop}), and so there is a constant $D_{\ref{T:sparsetheorem}}$ such that if also $|G|\geq D_{\ref{T:sparsetheorem}}m$ and $m\geq Dt$ then $G$ contains $H$ as a minor. We note $m\geq 0.3Ct$, and therefore taking $C>4D$ will suffice for this to hold.
      
      We can therefore assume $|G|\leq D_{\ref{T:sparsetheorem}}m$. In this case, $G$ is $(\epsilon/ D_{\ref{T:sparsetheorem}})|G| $ connected. Further, if $G$ has density $p$ (which must be at least $\epsilon / D_{\ref{T:sparsetheorem}}$ due to connectivity), then $e(G)/|G| \geq (\alpha + \epsilon)wt \geq p wt/ 2\sqrt{\log(1/(1-p+\epsilon/4))}$ by inequality~(\ref{eq:alphaproperty}). Taking $\eta = \epsilon/D_{\ref{T:sparsetheorem}}$, and using a value of $\eta/2$ in place of $\epsilon$, Theorem~\ref{T:ctdgamma} directly shows that if we take $d$ to be sufficiently large (depending on all previous constants), then $H$ is a minor of $G$. Taking $D$ large enough all of the above conditions are satisfied implies the result.
     
\end{proof}

  \section{Almost compatible partitions}

    \begin{definition}
    Let $H$ be a graph, and $t\in \mathbb{N}$. A partition $(V_h)_{h\in V(H)}$ of $V(G)$ is called \textsl{$t-$almost-$H$-compatible} if for all but at most $t$ edges $hh'$ of $H$, there is an edge in $G$ between $V_h$ and $V_{h'}$ (we say $V_h$ and $V_{h'}$ are adjacent, and write $V_h\sim V_{h'}$). 
    \end{definition}
    
    The aim of this section is to build an almost-$H$-compatible partition, which in the dense case we will be able to convert into a minor. This mimics the proof method of \cite{ThomasonWales}, although additional work is required here to handle vertices differently.

    \subsection{Compatible partitions with multiple classes}

\begin{theorem}\label{T:mpgl}

Let $a_1,b_1,...,a_r,b_r,l, n \geq l (a_1b_1+...+a_rb_r)$ be integers with each $a_i\geq 2$. Let $\omega \geq 2$, $\eta>0, \eta<p<1$ be constants such that $\omega\eta > 2r$. Let $G$ be a graph of density at least $p$ with at least $n$ vertices.

Then $V(G)$ has a partition $\mathcal{P}$, which can itself be decomposed into subpartitions $(\mathcal{P}_i)_{i=1}^r$, where $\mathcal{P}_i$ contains at least $a_i(1-2r/\omega\eta)$ parts, and the proportion of pairs of parts from $\mathcal{P}_i, \mathcal{P}_j$ respectively which have no edge between them is at most $4 r^2 6^{lb_ab_b}\omega^{lb_a} (\frac{q}{1-\eta})^{(1-\eta)l(l-1)b_ab_b}$, where $q = 1-p$.

\end{theorem}	
\begin{proof}

    We may assume $n = l(a_1b_1+\dotsc+a_rb_r)$ by replacing $G$ with a maximal density subgraph on that number of vertices --- at the end, we can redistribute any extra vertices arbitrarily among the parts. We also assume that $G$ has density exactly $p$ by applying the theorem with the actual density, then weakening our final result.
    
    Order the vertices of $G$ in non-increasing degree order, so that $d(v_1)\geq d(v_2)\geq \dotsc\geq d(v_n)$ and let $(q_i)$ be such that $d(v_i) = (1-q_i)(n-1)$. To simplify notation, if $v = v_i\in V(G)$, let $q_v = q_i$. Partition the vertices of $G$ into $l$ \textsl{blocks} $(B_i)_{i=1}^l$ of size $x = (a_1b_1+\dotsc+a_rb_r)$, where $B_i = \{v_{(i-1)x+1},\dotsc,v_{ix}\}$. 
    
    We will start by constructing $a_1$ parts for $\mathcal{P}_1, \dotsc,$ and $a_r$ for $\mathcal{P}_r$ randomly and one at a time as follows, with each part of $\mathcal{P}_i$ having $b_i$ vertices in each of the $l$ blocks for a total of $lb_i$ vertices. Let $X$ be the union of the vertex sets of parts already chosen, and suppose we are now picking a part $W$ for $\mathcal{P}_a$. $W$ will consist of a uniformly randomly chosen $b_a$ vertices from each of $B_1 \setminus X,\dotsc,B_l\setminus X$. Each of these sets has size at least $b_a$ by the definition of $x$, and so we can choose $W$ in such a fashion. We could equally performed this construction by picking a (random) partition of each block $B_i$ into $a_1$ subsets of size $b_1$, $a_2$ of size $b_2$ and so forth. We then form a part for $\mathcal{P}_a$ as a random size $a_a$ subset from each $B_i$. This also shows that having fixed (or conditioned on) one part $W$, each other part consists of uniformly randomly chosen elements of each $B_i\setminus W$. 
    
    Fix now some indices $a,b$ (where $a = b$ is permitted). We say that a vertex $v$ is \textsl{bad} for a set $W$ if $v$ has no neighbour in $W$, and $v\notin W$. For two disjoint sets $W,\widetilde{W}$ to be non-adjacent, $\widetilde{W}$ must consist entirely of vertices which are bad for $W$.
    For fixed $v$, what is the probability a part $W$ from $\mathcal{P}_a$ has $v$ bad for $W$? This means we have chosen $W$ as a subset of the $(n-1-d(v))= q_v(n-1)$ non-neighbours of $v$. If $W$ was instead a uniformly random $lb_a$ set, the probability of this would be at most $(q_v(n-1)/n)^{lb_a}$. We show this upper bound still holds for our blocked setup.
    
     Let $S$ consist of the $q_v(n-1)$ non-neighbours of $v$, and let  $S_i = S\cap B_i$. The probability that within $B_i$ we choose only non-neighbours of $v$ is at most $(|S_i|/x)^{b_a}$, and since our choices are independent we have an overall upper bound via the AM-GM inequality of \begin{align}
     \label{eq:blockeq1}
        \mathbb{P}(v\textrm{ is bad for }W) \leq \prod_i (\frac{|S_i|}{x})^{b_a} \leq (\frac{n-1}{n}q_v)^{lb_a}\leq q_v^{lb_a}
    \end{align}

    In particular,the expected number of vertices in block $i$ which are bad for $W$ is at most $x(q_{ix})^{lb_a}$, recalling that the $q_i$ are non-decreasing. We say that a part $W$ \textsl{rejects} block $B_i$ (where $i<l$) if there are more than $\omega x q_{ix}^{lb_a}$ bad vertices for $W$ in $B_i$ --- the probability of this event is at most $1/\omega$.
    
    We say that a part $W$ is \textsl{good} if it rejects fewer than $\eta(l-1)$ of the blocks $B_1,\dotsc,B_{l-1}$. The probability a part is good is at least $(1-1/\omega\eta)$.

    Suppose we fix some choice of good part $W$ from $\mathcal{P}_a$. Conditional on this choice of $W$, what is the probability a (distinct) part $\widetilde{W}$ from $\mathcal{P}_b$ is not adjacent to $W$? Within each block which is not rejected by $W$, we have at most $\omega x q_{ix}^{lb_a}$ remaining choices of vertex which are not adjacent to $W$, and $x-b_a$ remaining vertices to pick from. If we let $M(W)$ be the collection of indices $1\leq j \leq l-1$ for which $B_j$ is not rejected by $W$, with $m = |M(W)|$, we have
    \begin{align}
    \label{eq:blockineq2}
        \mathbb{P}(\widetilde{W}\nsim W) \leq \prod_{i\in M(W)}(\omega x q_{ix}^{lba}/(x-b_a))^{b_b}\leq (2\omega)^{mb_b}(\prod_{i\in M(W)}q_{ix})^{lb_ab_b}
    \end{align}
    Further, since the $(q_i)$ are non-decreasing, by applying the AM-GM inequality we deduce the following.
    \begin{align*}
        \prod_{i\in M(W)}q_{ix}^{1/m}\leq \frac{1}{m}\sum_{i\in M(W)}q_{ix} \leq \frac1{xm}\sum_{i\in M(W)}\sum_{j=1}^x q_{it + j} \leq\frac{1}{xm}\sum_{j=1}^{st} q_j \leq \frac{l}{m}q
    \end{align*}
    
   The above result combined with inequality~(\ref{eq:blockineq2}) implies the following.
    \begin{align*}
         \mathbb{P}(\widetilde{W}\nsim W) &\leq (2\omega)^{mb_b}(\frac{ql}{m})^{mlb_ab_b} \leq (2\omega)^{lb_b}(\frac{l}{l-1})^{l^2b_ab_b}(\frac{q}{1-\eta})^{(1-\eta)l(l-1)b_ab_b} \\&\leq 6^{lb_ab_b}\omega^{lb_b} (\frac{q}{1-\eta})^{(1-\eta)l(l-1)b_ab_b} 
    \end{align*}
    The second inequality above uses the bounds $(1-\eta)(l-1)\leq m \leq l-1$. We now modify the partitions described earlier to derive a final partition of $V(G)$ with the desired properties. The probability that, for some $i$, there are more than $2r/\omega\eta$ proportion of bad parts in $\mathcal{P}_i$ is less than $\frac1{2r}$ by Markov's inequality. The probability that there are more than $2r^2(6^{lb_ib_j}\omega^{lb_i} (\frac{q}{1-\eta})^{(1-\eta)l(l-1)b_ib_j})$ pairs of parts from $\mathcal{P}_i$, $\mathcal{P}_j$ respectively which are non-adjacent, but at least one of $\mathcal{P}_i$ and $\mathcal{P}_j$ is good is less than $\frac1{2r^2}$ again by Markov's inequality. Thus by a union bound, there is at least one choice of our partitions $(\mathcal{P}_i)$ for which all of these events do not occur. The result follows by deleting all bad parts from each $\mathcal{P}_i$, and redistributing unused vertices arbitrarily among the good parts so that we have a genuine partition of $V(G)$.
\end{proof}

\subsection{Building almost-compatible partitions}

In this subsection, we make some choices of parameters to obtain almost-compatible partitions for our desired settings. In the next section, we will convert these into proofs of the dense cases. We will require the following inequality; for a proof, see \cite[Lemma 2.3]{ThomasonWales}
\begin{align}
\label{eq:logineq}
    \frac{\log(x+\epsilon)}{\log x}\leq (1-\epsilon)^2, \mathrm{for\,\,any\,\, } 0<\epsilon<\frac12 ,\, 0< x \leq 1-\epsilon
\end{align}

\begin{lemma}
\label{T:bipblock}
    Let $\epsilon>0$. Then there is a constant $C = C_{\ref{T:bipblock}}$ such that if $t>0$, $C < f < \log t$, and $H = K^*_{ft/\log t,t}$ the following holds.
    
    Suppose that $G$ is a graph with density at least $p+\epsilon$, where $\epsilon<p<1-\epsilon$,  and $n \geq 2\sqrt{f}t/\sqrt{\log(1/q)}$ vertices for $q = 1-p$. Then $G$ has a $t^{1-\epsilon/4}$-almost-$H$-compatible partition.
\end{lemma}
    
    \begin{proof}
    
    The core of this proof is an application of Theorem~\ref{T:mpgl}, although work is required to choose suitable parameters. We will be taking $r = 2$, and using parts from $\mathcal{P}_1$ to correspond to vertices in the class of size $ft/\log t$ and from $\mathcal{P}_2$ to the vertices of the other class. Our density for the application of Theorem~\ref{T:mpgl} will be $p+\epsilon$ rather than $p$. Let $M,\omega$ be large constants yet to be determined, and $\eta,\delta$  small constants which will not depend on $C$.
    
    Let $a_1 = \floor{(1+\delta)ft/\log t}$, $a_2 = \floor{(1+\delta)t}, b_1 = \floor{(\log t)/M}, b_2 = \floor{f/M}$ and $l= \floor{(1-\delta)M/\sqrt{f\log(1/q)}}$ be the remainder of the parameters. These were chosen according to some `optimal weighting' to maximise $b_1b_2$ subject to an upper bound on $\sum a_ib_i$.
    
        It is easily seen that $l(a_1b_1 + a_2b_2)\leq n$ and so the conditions of Theorem \ref{T:mpgl} are satisfied --- hence $G$ has a partition $\mathcal{P}_1\cup\mathcal{P}_2$ satisfying the hypotheses therein. We assume that $l,a_1,a_2,b_1,b_2 > 4/\delta$, since we can make both $f$ and $t$ arbitrarily large, and will choose $M$ so that this occurs.
        
        If $\rho$ is the density of non-adjacent parts between $\mathcal{P}_1$ and $\mathcal{P}_2$, we have
        \begin{align*}
            \log \rho \leq 4 + \frac{\log 6 (\sqrt{f}(\log t) / M) + (\log \omega)(\log t)/\sqrt{f}}{\sqrt{\log(1/q)}} +\\ (1-\eta)l(l-1)\floor{(\log t)/M}\floor{f/M}(\log(q-\epsilon)-\log(1-\eta))
        \end{align*}
        
       We note  $l-1 \geq (1-\delta)(1-\delta)M/\sqrt{f\log(1/q)}$. Recalling Inequality~(\ref{eq:logineq}) and that $\epsilon < p <1-\epsilon$, for sufficiently small $\eta$ we have $\frac{\log(q-\epsilon) - \log(1-\eta)}{\log q} > \frac{1}{1-\epsilon}$. We therefore deduce the following.
        
        \begin{align}
        \label{Eq:bippartmidproof}
            \log k \leq 4 + \frac{(\log 6) (f\log t / M) + \log(\omega)\log t}{\sqrt{f\log(1/q)}} - \frac{(1-\eta)(1-\delta)^3}{(1-\epsilon)}\log t
        \end{align}
        
        We now fix some choices of $\delta$ and $\eta$ sufficiently small that the final summand of Equation~(\ref{Eq:bippartmidproof}) is at most $-(1+\epsilon/2)\log t$.

        We have $|\mathcal{P}_i|\geq a_i(1-2r/\omega\eta)$. Choose (and fix) some $\omega$ so that \\$(1+\delta)(1-\delta/2)(1-4/\omega\eta) > 1$; we then have at least $ft/\log t$ parts in $\mathcal{P}_1$, and $t$ in $\mathcal{P}_2$ as given by Theorem~\ref{T:mpgl}. In order that all of our prior claims are satisfied, and in order to bound $\log k$ suitably we will require $M = o(\log t) = o(f) = \omega(\sqrt{f})$ --- we make a choice $M = f^{3/4}$ noting that $f\leq \log t$ by assumption.
        
        We choose $C$ sufficiently large that for any $\log t\geq f \geq C$, the first two terms of Inequality~(\ref{Eq:bippartmidproof}) are each at most $\epsilon/12 \log t$ --- hence $\rho\leq t^{-1-\epsilon/3}$. Picking randomly some $f(t)t/\log t$ parts from $\mathcal{P}_1$, and $t$ from $\mathcal{P}_2$,the expected number of pairs of picked parts, the first from $\mathcal{P}_1$ and the second from $\mathcal{P}_2$, is at most  $ft^2\rho/\log t\leq t^2\rho$. Thus with probability more than $\frac12$, at most $2t^{1-\epsilon/3}$ such parts are non-adjacent.
        
        Let $\rho'$ be the density of non-adjacent pairs in the left hand class. Since $b_2\leq b_1$ we immediately deduce $\rho'\leq \rho \leq t^{-1-\epsilon/3}$. Adding these together we deduce our desired partition exists.
    \end{proof}


    \begin{theorem}
    \label{t:weightingpartition}
    Let $r> 0$, $\epsilon<p<1-\epsilon$ and suppose $H$ is a graph with $t$ vertices and $td$ edges for some $d>1$. Let $(P_i,w_i)_{i=1}^r$ be a weighted partition of $V(H)$ such that
    \begin{align}\label{eq:gammainequality}\sum_{i\leq j} d^{-w_iw_j}e(P_i,P_j)/t \leq 1\end{align}
    
    Let $w = \sum_i |P_i|w_i$, and suppose $w \geq 2^{10}\epsilon^{-3} t /\sqrt{\log_{1/q}d}$. Let $G$ have density at least $p+\epsilon$, where $p = 1-q$, and $n\geq w \sqrt{\log_{1/q}d}$ vertices. Then $G$ contains a $64r^6t/\epsilon$-almost-$H$-compatible partition.
    \end{theorem}

    \begin{proof}
	Our proof broadly mimics that of Theorem~\ref{T:bipblock}, although the calculations are more involved. Let $\delta = \epsilon/8$, and $l = \ceil{(\delta^2/8)w\sqrt{\log_{1/q}d}/t}$. The lower bound on $w$ in the statement implies $l > 2/\epsilon$, and so $l< 2w(\delta^2/8)\sqrt{\log_{1/q}d}/t$.

        Replace each $w_i$ by a new weight $w'_i$ so that \begin{align}
        \label{eq:generalweighting}
        \ceil{(1-\delta)w_i\sqrt{\log_{1/q}d}/l} = (1-\delta)w'_i\sqrt{\log_{1/q}d}/l\end{align} We observe that $(P_i,w'_i)_{i=1}^r$ still satisfies the Gamma inequality~(\ref{eq:gammainequality}) since we have only increased the weights. We will be picking parts for each vertex of $H$ randomly, with the size of that part proportional to the weight. Let $a_i = \ceil{|P_i|(1+\delta/4)}$, and $b_i = (1-\delta)w'_i\sqrt{\log_{1/q}d}/l$.

        We now seek to apply Theorem~\ref{T:mpgl}, with $\eta = \epsilon^2/4$, and $\omega = 8r/\eta\delta$ (and hence $(1-2r/\omega\eta)(1+\delta/2) > 1$ ) --- note $G$ has density at least $p+\epsilon$ so we use this in the application. We would like to assume that all the $a_i$ and $b_i$ are sufficiently large, so that we can ensure the earlier rounding of $a_i$ and $b_i$ does not make us violate the condition $|G|\geq\sum la_ib_i$. However in general this may not be the case  --- it is valid to take a single vertex of $H$ as a part $P_i$.
 
  We resolve this problem by taking all parts with $|P_i|< 4/\delta$, and replace $w'_i$ with $0$ on the corresponding indices (call such indices \textsl{terrible}). In particular, we now satisfy $a_i \leq (1+\delta/4)^2 |P_i|$ for all $i$. We now use this to show $|G|$ is sufficiently large to apply Theorem~\ref{T:mpgl} with this weighting --- the additional factor of $2tl < (\delta^2/2) w\sqrt{\log_{1/q}d}$ arising since $w'_i\sqrt{\log_{1/q}d} \leq w_i\sqrt{\log_{1/q}d} + l/(1-\delta)$
   
        \begin{align*}
            \sum la_ib_i &\leq \sum_i (1-\delta)(1+\delta/4)^2 |P_i| w'_i\sqrt{\log_{1/q}d} \\&\leq \sum_i (1-\delta/3) |P_i| w_i\sqrt{\log_{1/q}d} + 2tl
            \leq w\sqrt{\log_{1/q} d}
        \end{align*}
        
        We now apply Theorem~\ref{T:mpgl} to $G$ with a suitable choice of $\omega$ and $\eta$. We will obtain our almost compatible partition by picking a subset of the partition $\mathcal{P}_i$ from that theorem to represent the vertices of $P_i$ in $H$. Suppose that $H$ has $td$ edges, of which $td^{D_{ij}}$ are between vertices in $P_i$ and in $P_j$.

        We restrict attention only to parts $P_i$ which are not terrible; for the at most $4r/\delta$ vertices in terrible parts, we accept the at most $4rt/\delta$ total non-edges incident with them. We only analyze pairs with $D_{ij}>0$; even if all edges between pairs with $D_{ij}\leq 0 $ are bad, this is fewer than $r^2t$ edges, which we add at the end. If $\rho$ is the density of non-adjacent pairs between such a pair $i,j$, where we order so $w'_i\leq w'_j$, we have the following bound from the application of Theorem~\ref{T:mpgl}.
        \begin{align*}
            \log \rho \leq \log(4r^2) + (\log 6)(lb_ib_j) + (\log \omega)(lb_i) +\\ (\log(q-\epsilon) - \log(1-\eta))(1-1/l)(-\log d / \log q)w'_iw'_j(1-\delta)^2 
        \end{align*}
        
        We would now like to bound this quantity more simply; in order to obtain the result in this theorem we will require $\log \rho \leq -w_iw_j\log d$ and thus $\rho \leq d^{-w_iw_j}$.
        
        We control the final two summands on the first line one at a time, and the second line as a whole. The quantity on the second line can be bounded above by $-(1+\epsilon/3)w'_iw'_j\log d $ by an application of Inequality~\ref{eq:logineq} as in the proof of Lemma~\ref{T:bipblock}.
        
        We now handle the final term on the first line, we have $lb_i \leq l(w'_i\sqrt{\log d} / l)$. We note that $w'_i\sqrt{\log_{1/q} d}\geq l/(1-\delta)$ by Equation~\ref{eq:generalweighting} and since $p>\epsilon$. In particular, since also $p<1-\epsilon$, $w'_i\sqrt{\log d} > 9r/\epsilon $ by the bound on $l$. We therefore deduce $lb_i\leq  w'_i\sqrt{\log d}\leq \epsilon/9 w_iw_j\log d$.
        
        The second term $lb_ib_j \leq (1/l)(w_iw_j\log_{1/q}d)\leq \epsilon/9 w_iw_j\sqrt{\log d}$, and finally the first term can be bounded using $\log(4r^2)\leq \log(4r^2) (\epsilon/9)^2 w'_i\sqrt{\log d}w'_j \sqrt{\log d} $.
        
       We therefore deduce $\rho \leq 4r^2 d^{\epsilon D_{ij}/9} t^{-(1+\epsilon/6)w_iw_j} \leq 4r^2d^{-(1+\epsilon/9)D_{ij}}$.

    For each $i$, independently choose a random assignment of vertices from $P_i$ to parts in $\mathcal{P}_i$. If we do this, we expect $\rho$ proportion of the edges of $H$ between $P_i$  and $P_j$ to correspond to non-adjacent pairs, where $\rho$ depends on $i$ and $j$. Take some choice of assignments for which no pairs $i,j$ exceed $r^2$ times the expected number of non-adjacent pairs corresponding to an edge of $H$ between $P_i$ and $P_j$ (this is possible by Markov's inequality). 
    
    As mentioned above, there are at most $4rt/\delta$ edges of $H$ incident with a part $P_i$ which is terrible, and at most $2r^2t$ edges are between pairs with $D_{ij} < 0$. Consider some pair $i,j$ for which neither $P_i$ nor $P_j$ is terrible, and $D_{ij}\geq 0$. In this choice of assignments, there are at most $4r^4td^{-\epsilon/9 D_{ij}}$ non-adjacent pairs corresponding to an edge of $H$ between $P_i$ and $P_j$ using the above bound --- in particular by inequality~\ref{eq:gammainequality} this is at most $4r^4$. Combining all of the above analysis by summing over pairs, at most $8r^6t/\delta$ edges of $H$ correspond to non-adjacent pairs, which completes our proof.

    \end{proof}
    
    \begin{cor}
    \label{t:gammapartition}
    Let $r>0$ and $\epsilon < p <1-\epsilon$, and let $H$ be a graph with $t$ vertices and $td$ edges for some $d>1$. 
    Suppose that $G$ has density at least $p+\epsilon$, and $n\geq \gamma_r(H)t\sqrt{\log_{1/q} d}$. Further, suppose that $\gamma_r(H)\sqrt{\log d} > 2^{10}\epsilon^{-3}\sqrt{\log(1/q)} $
    Then $G$ has a $64r^6t/\epsilon$-almost-$H$-compatible partition
    
    \end{cor}
    
    \begin{proof}
        Let $(P_i,w_i)$ be some weighted partition of $H$ into $r$ parts satisfying inequality (\ref{eq:gammainequality}) with $\sum_i |P_i|w_i$ minimal. Such a partition exists, because we are optimising over a finite number of partitions, and for each such partition we are solving a closed optimisation problem to pick an optimal weighting. In particular, we have $\sum_i |P_i|w_i = t \gamma_r(H)$. We now apply Theorem~\ref{t:weightingpartition} to this weighted partition, with all other parameters as in the theorem.
    \end{proof}

    \section{Connector, Projector}
    \label{S:connproj}
    We will now borrow some small helper sets from \cite{ThomasonWales} that will allow us to turn an almost-compatible partition into a minor, under the condition that $G$ has suitable connectivity. The following theorem is a combination of \cite[Theorem 2.7 and Lemma 2.8]{ThomasonWales} --- the set $CP$ here being the union of $C$ and $P$ from those theorems.
    
    \begin{theorem}
    \label{T:connectpartition}
    \label{T:connectorprojector}
    Given $\eta > 0$ there exists $D = D_{\ref{T:connectpartition}}(\eta)$ such that if $G$ is a graph with $|G|\geq D$ and $\kappa(G)\geq 8\eta |G|$, for each $R\subset V(G)$ with $|R|\leq |G|/D$ there is a subset $CP$ of $V(G)$, where $|CP|\leq 4\eta|G|$ such that the following holds.
    
    Let $(V_r)_{r\in R}$ be a partition of $V(G)-R-CP$ into $|R|$ parts, and \\$F\subset\{\,rs\,:\,r,s\in R\}$ a collection of pairs from $R$ with $|F|\le |R|/\eta$. Then there are disjoint subsets $(U_r)_{r\in R}$ of $V(G)$ with the following properties.
    \begin{enumerate}
        \item{}$V_r\cup \{r\}\subset U_r$ for all $r\in R$, 
        \item{}$G[U_r]$ is connected for all $r\in R$, and
        \item{} there is a $U_r-U_s$ edge for every pair $rs\in F$.
\end{enumerate}
    
    \end{theorem}
    
We can now prove the dense case (i.e. Theorems \ref{T:ctdbipartite} and \ref{T:ctdgamma}) using this result; the theorems are restated for convenience.

\ctdtheorem*

\begin{proof}

    Let $f,t$ be as in the statement, $q = 1-p$ and $H = K^*_{ft/\log t,t}$. Reduce $\eta$ if necessary so that $20\eta < \epsilon$ and also $(1-5\eta) \geq \sqrt{\log(1/q)/\log(1/(q-\epsilon/4))}$ --- note that $G$ still satisfies the hypotheses with this reduced value. Let $R$ be a set of roots labelled by the vertices of $H$. Take $C$ large enough that $2\sqrt{C/\log(1/\epsilon)} \geq 2/\eta$. Then since $C\leq f\leq \log t$ we have $|R|\leq |G|/D_{\ref{T:connectpartition}}(\eta)\leq \eta|G|$. 
    
    By Theorem~\ref{T:connectorprojector}, assuming $C$ is sufficiently large, there is a subset $CP$ with certain properties which we will use later. The subgraph $G-CP-R$ has density at least $p+\epsilon/2$ since we have only removed at most $\epsilon/4$ proportion of vertices.
    
     However, $G-CP-R$ may have fewer than $2t\sqrt{f/\log(1/q)}$ vertices, which prevents a naive application of Lemma~\ref{T:bipblock}. Since we only removed at most $5\eta |G|$ vertices, we have $|G-CP-R|\geq n(1-5\eta) \geq 2t\sqrt{f/\log(1/(q-\epsilon/4))}$ by inequality~(\ref{eq:logineq}), and so we can apply Lemma~\ref{T:bipblock} with $p$ replaced by $p+\epsilon/4$, and $\epsilon$ replaced with $\epsilon/4$, since the density of the subgraph $G-CP-R$ is at least $((p+\epsilon/4) + \epsilon/4)$. This application requires $D > C_{\ref{T:bipblock}}(\epsilon/4)$.
    
    As a consequence, $G-CP-R$ has a $t^{1-\epsilon/16}$-almost-$H$-compatible partition $(V_h)$. Let $F$ be the collection of adjacent pairs in $H$ whose corresponding subsets are non-adjacent. Using the properties of the set $CP$ from Theorem~\ref{T:connectpartition} with the collection of bad pairs $F$ the result follows since $|F| \leq t^{1-\epsilon/16} \leq t^{-\epsilon/16}|R|\leq |R|/\eta$ provided $C$ is large enough.
\end{proof}

\gammactdtheorem*

\begin{proof}
 Reduce $\eta$ if necessary so that $(1-5\eta)\geq \sqrt{\log(1/q)/\log(1/(q-\epsilon/4))}$ for all $1\geq q \geq \epsilon$ (where $q$ will denote $1-p$), and also  $\epsilon/(2^8r^6)>\eta$. Let $R$ be a set of roots labelled by the vertices of $H$, and we will choose $T > D = D_{\ref{T:connectpartition}}(\eta)$ so $|R|\leq |G|/D\leq \eta |G|$. By Theorem \ref{T:connectorprojector}, there is a subset $CP$ of size at most $4\eta|G|$ with certain properties which we use later.
 
Since we delete at most $5\eta$ vertices, $G-CP-R$ has density at least $p+\epsilon/2$, and further has at least $n(1-5\eta)\geq w\sqrt{\log_{1/(q-\epsilon/4)}d}$ vertices. In particular, we can apply Theorem~\ref{T:mainweightthm}, with the parameter $p$ from that theorem replaced with $p + \epsilon/4$ , and $\epsilon$ with $\epsilon/4$ provided we take $T > 2^{16}\epsilon^{-3}$ to deduce $G-CP-R$ has a $256r^6t/\epsilon$-almost-$H$-compatible partition $(V_h)$. Let $F$ denote the collection of adjacent pairs in $H$ whose corresponding subsets are non-adjacent, and note $|F|\leq 256r^6t/\epsilon \leq |R|/\eta$. Using the property of $CP$ with this choice of $F$, we obtain sets $U_h\supset V_h$ which form a $H$ minor rooted at $R$, as desired.
\end{proof}

\section{The sparse case}

Even in the sparse case, $G$ has the properties of Lemma~\ref{L:enkprop}, and hence each edge is in many triangles --- equivalently for any vertex $v$, $G[N(v)]$ has reasonable minimum degree. Since $G$ has known average degree, this allows us to find small subgraphs of reasonable density.  However, we cannot directly find an $H$ minor in these graphs (we have given away too much).

Our proof strategy is to find many different parts of $H$ in different subgraphs, and somehow join them to form an $H$ minor. We will have a problem with this approach if when attempting to find these subgraphs, we cannot prevent them from significantly overlapping. But in this case, we will be able to find a dense bipartite subgraph, and therefore use the following lemma to directly find a $H$ minor.
 \begin{lemma}
 \label{T:minorinbipartite}
    Let $G = (A,B)$ be a bipartite graph, and $m,C,\eta,\epsilon$ parameters such that $G$ has the following properties.
    \begin{itemize}
        \item $|B|\leq Cm$
        \item For all $a\in A$, $d(a)\geq \eta m$
    \end{itemize}
    Then there is a constant $D = D_{\ref{T:minorinbipartite}}(\eta,C,\epsilon)$ such that if further $|A| > Dm$, then $G$ has a minor $\widetilde{G}$ on at least $\eta m$ vertices with minimum degree at least $(1-\epsilon-1/m)|\widetilde{G}|$.
    \end{lemma}
    
    \begin{proof}
        We start with a graph $G = G_0$ on vertex set $A\cup B$. Given a graph $G_i$, and vertex partition $A_i\cup B$, choose some $a\in A_i$, and form a graph $G_{i+1}$ on vertex partition $(A_i-a)\cup B$ by contracting $a$ to a vertex of minimum degree in $G_i[N(a)]$ (and deleting any multiple edges). Note each such $G_i$ is a minor of $G$. There are two cases: either the following property holds, or it does not.
        \begin{align}
        \label{eq:bipartite_graph_minor}
            \delta(G_i[N(a)])\geq (1-\epsilon-1/m)d(a)
        \end{align}
        
        If this fails, $G_{i+1}[B]$ has at least $\epsilon\eta m$ more edges than $G_i[B]$. Since all $G_i$ are simple graphs, all $G_i[B]$ must have at most $|B|^2/2$ edges. In particular, if $|A| \geq (C^2/\epsilon\eta) m$, at some step we must satisfy property (\ref{eq:bipartite_graph_minor}) --- else $e(G_{|A|}(B)) \geq C^2m \geq |B|^2/2$. Taking $D> C^2/\epsilon\eta$, the result follows by choosing our minor as $G_i[N(a)]$.
    \end{proof}
    
    We finish this section by proving the sparse case, and hence completing the proof of all our upper bounds. For convenience, we restate the theorem here.
    
    \sparsetheorem*

    \begin{proof}
    
        Suppose that $G$ is a graph with properties (1)-(4). We note that since $G$ is connected, every vertex has at least one neighbour. Further, since every edge lies in at least $m-1$ triangles, we deduce $|N(v)|\geq m$ and $\delta(G[N(v)])\geq m-1$ for every vertex $v\in V(G)$. We start by finding one at a time disjoint non-empty subsets $S_0,\dotsc,S_k$,  where each $|S_i|\leq 6m$, and $G[S_i]$ has minimum degree at least $5m/6$.  Suppose we have already found $S_0,\dotsc,S_r$, and call their union $B$ (note $|B|\leq 6km \leq |G|/3$ provided $D>6k$). 
        
        Let $A$ be all vertices not in $B$, and with degree at most $6m$. Since $e(G)\leq m|G|$, by Markov's inequality it follows $|A|\geq |G|/3$. Suppose first that every $a\in A$ has at least $m/6-1 \geq m/7$ neighbours in $B$. In this case, we apply Lemma \ref{T:minorinbipartite} with $\eta = 1/7, C=6k$ and the current $\epsilon< 1/6$ to the bipartite subgraph of $G$ induced by $A$ and $B$.
        
        This implies that $G$ has a minor $M$ with minimum degree at least $5|M|/6$ and at least $m/7$ vertices. $M$ has connectivity at least $|M|/3$ by the minimum degree condition, so by hypothesis $M$ (hence $G$) contains $H$ as a minor. Since if this happens we are done directly, we assume this does not happen and so there is some $a\in A$ with fewer than $m/6 -1$ neighbours in $B$. Then $S_{r+1} = N(a)\setminus B$ has the required properties. We can thus assume we have built the subsets $(S_i)_{i=0}^k$, and now modify them to also have good connectivity.

        Fix some set $S_i$ with the properties above. If $G[S_i]$ is at least $(m/40)$-connected, we set $T_i = S_i$. Otherwise, fix some choice of cutset $C$ of size at most $m/40$. Then $G[S_i-C]$ has a connected component $S_i'$ of size at most $3m$, and $G[S_i']$ has minimum degree at least $m(5/6-1/40)$. 
        
        If $S_i'$ is $(m/40)$-connected, set $T_i = S_i'$, and otherwise repeat the above procedure. After 3 iterations, we have either constructed $T_i$, or are left with a set $S_i'''$ of size at most $3m/4$ and minimum degree at least $m(5/6 - 3/40) > 3m/4$. This is a contradiction. We can therefore, for all $i$, construct subsets $T_i\subset S_i$ such that $G[T_i]$ has minimum degree at least $3m/4$ and connectivity at least $m/40$.
        
        The $T_i$ are now set up for building minors, and it remains to build paths between them. Recalling that $G$ is $D|H|$ connected by Property (3), we can apply Menger's theorem to find at least $\sum |H_i|$ paths from $T_0$ to an (arbitrary) subset $R$, which we write as a union of subsets $R_i\subset T_i$ of $|H_i|$ vertices. By replacing these paths with subpaths, we can assume each path meets $T_0$ in exactly one vertex.
        
        Initially, these paths may consume almost all of the vertices in the $T_i$. We will use the structure of $H_i$ to fix this case. Pick some index $i$ for which the paths use more than $81k|H|$ vertices in $T_i$ (and if no such $i$ exists, we terminate). For each path $P$ which meets $T_i$, there is a first time it intersects $T_i$, call it $x_P$, and a last time, call it $y_P$. 
        
        We will now, for each such path $P$ in turn, find a path of length at most 81 inside $T_i$ between $x_P$ and $y_P$, and replace the part of $P$ between $x_P$ and $y_P$ with this new path. We will do this in such a way all these new paths are disjoint. This procedure only reduces the number of vertices in the paths intersecting any other $T_j$, and so we can sequentially fix the $T_i$ in a terminating procedure. It remains to show we can do this for all (at most $k|H|$) paths which intersect $T_i$. For this, we use the following lemma.
        
        \begin{lemma}
        \label{L:shortpaths}
           (Short Paths) Let $G$ be a graph with connectivity at least $\eta |G|$. Then between any given pair of endpoints, $G$ contains at least $\eta|G|/2$ internally vertex disjoint paths of length less than $2/\eta + 1$.
        \end{lemma}
           \begin{proof}
            By Menger's theorem, $G$ has at least $\eta|G|$ internally vertex disjoint paths between the endpoints. At least $\eta|G| /2$ of those must have length less than $(2/\eta) + 1$, since otherwise the at least $\eta|G|/2$ longer paths each have at least $2/\eta$ internal vertices, which requires more than $|G|-2$ non-endpoint vertices since the paths are internally disjoint.
        \end{proof}
        
        Applying Lemma~\ref{L:shortpaths} with $\eta = 1/40$, $T_i$ contains at least $(1/80)(3m/4) \geq m/200$ internally vertex disjoint paths of length at most 80 (and hence with at most 81 vertices) between $x_P$ and $y_P$. Consider the set $X$ of all $x_Q,y_Q$ for paths $Q\neq P$, as well as all internal vertices of previously constructed paths. We have $|X|\leq 81k|H|$, and in particular $X$ can only meet at most $81k|H|$ of our $m/200$ paths. Provided we take $D>16400k$, at least one of the $m/200$ paths does not meet $X$. We choose such a path arbitrarily as the new path. Continuing in this fashion builds the desired new collection of paths, and so we can assume each $T_i$ meets the common vertex set $Y$ of our paths, in at most $81k|H|$ places.
        
        Let $T_i' = (T_i \setminus Y) \cup R_i$. Then $G[T_i']$ still has minimum degree at least $\epsilon m$ and connectivity at least $\epsilon |T_i|$, provided $81k|H| \leq m/500$ (recall $R_i\subset T_i$, and deleting $\eta m$ vertices can only decrease minimum degree and connectivity by $\eta m$). Thus $G[T_i']$ is $H_i$-minor-prevalent, and in particular contains a $H_i$ minor rooted at $R_i$ in any way we choose. Let the vertex subsets of such a $H_i$ model be $(V^i_h)_{h\in V(H_i)}$.
        
        Suppose $h\in V(H)$. We form a subset $V_h$ consisting of:
        \begin{itemize}
            \item All subsets $V^i_h$ from the rooted $H_i$ minors (where $h \in V(H_i)$)
            \item All paths from $T_0$ to the roots lying in $V^i_h$
        \end{itemize}
        
       Having done so, $V_h$ now contains at most $k$ connected components (one for each index $i$), each component intersects $T_0$, and $|V_h\cap T_0| = k$. Applying Lemma~\ref{L:shortpaths} to $G[T_0]$, between each pair of vertices in $G[T_0]$ there are at least $m/200 \geq 81k|H|$ internally vertex disjoint paths of length at most 80 between them. In particular, we can sequentially pick disjoint such paths to connect each subset $V_h: h\in H$, and we add these paths to $V_h$ --- this uses at most $(k-1)|H|$ paths overall.
       
       It remains to show the $V_h$ form an $H$ model. By construction, each $V_h$ is connected. For each edge $hh'\in E(H)$, there is some index $i$ for which $hh'\in E(H_i)$. In particular, there must be an edge of $G[T_i]$ from $V^i_h$ to $V^i_{h'}$. But these form subsets of $V_h$, $V_{h'}$ respectively, and so there is an edge of $G$ between $V_h$ and $V_{h'}$. Since this holds for all edges of $H$, we have constructed our model and hence completed the proof.
    \end{proof}

\section{Many non-adjacent sets in random graphs}

We now turn to lower bounds. The statement `$H$ is not a minor of $G$' can be expanded as `for all $(V_h)_{h\in V(H)}$ disjoint connected non-empty subsets of $G$, there is some pair $h\sim_Hh'$ such that $V_h\nsim_GV_{h'}'$ '. We will be proving a stronger statement by removing the connectedness constraint.

To prove a result for almost all $H$, we will find a graph $G$ which has many non-adjacent pairs of subsets for any choice $(V_h)$ of $|H|$ subsets --- in particular, it is very likely we will find one non-adjacent pair among those which also correspond to edges of $H$. The following lemma provides such a graph, but is not of itself sufficient to provide a good lower bound; this is resolved in the next section.

\begin{lemma}
\label{L:basegraph}
Let $0<p<1$, $0<\epsilon$. There exists a constant $D = D_{\ref{L:basegraph}}(p,\epsilon)$ such that for all $d>D$ there is a graph $G = G_{d,p,\epsilon}$ with $d$ vertices and density at least $p$, with the following property.

Let $l = \sqrt{\log_{1/1-p}d}$, and $0\leq x \leq l$.
Let $A_1,\dotsc,A_s; B_1,\dotsc,B_s$ be disjoint subsets of $G$, such that $|A_i||B_j|\leq xl^2$ for $1\leq i,j\leq s$. Then provided $s \geq d^{\epsilon+x}$, there are at least $\frac{1}{2}d^{-x}s^2$ pairs $(A_i,B_j)$ which are non-adjacent (i.e. have no edge between them in $G$)
\end{lemma}

\begin{proof}
    Let $l = \sqrt{\log_{1/1-p}d}$. Since $|A_i||B_j|$ is an integer, and reducing $x$ for fixed $|A_i||B_j|$ makes the result stronger, we only need consider each of $x = 0,1/l^2,\dotsc,\floor{l^3}/l^2$ --- the $x = 0$ case being trivial, so we also assume $x\geq 1/l^2$.
    
    Let $G \sim G(d,p)$ be a random graph on $d$ vertices, with edges present independently with probability $p$. We will show that with positive probability, $G$ has the required properties.
    In \cite{BinomialOneQuarter}, it is shown the probability $G(d,p)$ has density at least $p$ is at least $1/4$, provided that $\binom{d}{2}p \geq 1$. For our result to remain self-contained, we give a sketch of a stronger result --- though we require that $d$ is large.
    Let $X \sim \mathcal{B}(\binom{d}{2},p)$ be the random variable $e(G)$. Let $\mu = \mathbb{E}(X)$, and $\sigma^2 = \mathrm{Var}(X) = \binom{d}{2}p(1-p)$. By the Central Limit Theorem, $\frac{X-\mu}{\sigma}$ converges in distribution to a standard normal $N(0,1)$ random variable (as $d$ increases). Let $z$ be such that $\Phi(z) = 2/3$. Then $\mathbb{P}(X\leq \mu + z\sigma)$ converges to $2/3$, and in particular for $d>D$ sufficiently large, with probability at least $1/4$ the random variable exceeds its mean.

    We will require the Chernoff bound in the following form, which follows from \cite{AlonSpencer}, Theorems A.1.13 and A.1.11 .
    \begin{lemma}[Chernoff bound]
    \label{L:chernoff}
        Let $(X_i)_{i=1}^n$ be independent random variables, taking values in $\{0,1\}$, such that $\mathbb{E}(\sum X_i)= \mu$. Let $X = \sum X_i$. Then\\ $\mathbb{P}(X>a+\mu)\leq \exp(-a^2/2\mu + a^3/2\mu^2)$ holds, and in particular so do the following. 
        \begin{align}
        \label{eq:chernoff1}
            \mathbb{P}(X>a+\mu)\leq \exp(-a^2/4\mu) \,(\mathrm{provided}\,\, a<\mu/2)\\
        \label{eq:chernoff2}
            \mathbb{P}(X<\mu-a) \leq \exp(-a^2/2\mu)
        \end{align}
        Moreover, since these bounds are monotonic in $\mu$, inequality (\ref{eq:chernoff1}) holds under the weaker condition $\mathbb{E}(\sum X_i) \leq \mu$, and inequality (\ref{eq:chernoff2}) for $\mathbb{E}(\sum X_i) \geq \mu$.
    \end{lemma}
    
    Fix now some choice of $x$, $(A_i),(B_j)$. Since each edge is present independently at random, $\mathbb{P}(\textrm{there is no }\hspace{1pt}A_i - B_j\textrm{ edge}) = (1-p)^{|A_i||B_j|}\geq d^{-x}$. In particular, the expected number of non-adjacent $A-B$ pairs is at least $s^2d^{-x}$
    
    Let $X_{ij}$ be the indicator variable for $A_i$ and $B_j$ being non-adjacent. Since our sets are disjoint, these depend on different edges and so are independent. In particular, we can apply the Chernoff bound with $a = \frac{1}{2}d^{-x}s^2$ obtaining
    \begin{align*}
        \mathbb{P}(\sum X_{ij} \leq \frac{1}{2}d^{-x}s^2) \leq \exp(-\frac{1}{8}d^{-x}s\cdot s) \leq \exp(-\frac{1}{8}d^{\epsilon}s)
    \end{align*}

    To finish, we bound the number of choices of $(A_i)$ and $(B_j)$. We remark it suffices to consider the case where all $A_i$ and $B_j$ are non-empty, and in particular $|A_i|\leq l^3$, $|B_j|\leq l^3$. There are at most $\binom{\binom{d}{l^3}}{s}^2\leq d^{2l^3s}$ choices of (non-empty) sets satisfying these constraints. Therefore, for each $x$ and $s$, the probability that some choice of $(A_i)$ and $(B_j)$ have too few non-adjacent pairs is at most $\exp(-\frac{1}{16}d^{\epsilon}s)$ since $d^{2l^3s}\leq \exp(\frac{1}{16}d^{\epsilon}s)$ for large $d$. We can sum over the at most $l^3$ nonzero values of $x$, and the at most $d$ choices for $s$ (each of which is at least $d^{\epsilon}$, and deduce that except with probability at most $\exp(-\frac{1}{32}d^{2\epsilon})$ we have the desired property for all choices of $s,x$. Taking $D$ large enough, this probability is less than $\frac14$, and hence with positive probability $G$ satisfies all of the hypotheses of the theorem; in particular one such graph exists.

\end{proof}

We will make use of Lemma~\ref{L:basegraph} in the following section. We finish this section by proving Theorem~\ref{T:gammafamilymatch} from Section~2, again restated for convenience.
\gammafamilymatch*

\begin{proof}
    Recalling the discussion from Section~2, the final two inequalities are easily seen to hold for all elements of our family from the definitions provided $D$ is sufficiently large, so we need only focus on the first. We seek to prove this lower bound on $\gamma(H)$ for almost all elements of the class.

    Suppose that $H\in \mathcal{D}(\parts,\admat)(t,d)$ has $\gamma(H)\leq (1-\epsilon)\gamma(\parts,\admat)\sqrt{\log d / \log t}$. Let $w$ be a weight function with $\sum_{uv\in E(H)}t^{-w(u)w(v)}\leq t$, with average weight $\frac1t\sum_{u\in H}w(u) = \gamma(H)$. Consider the below modified weighting.\begin{align*}w'(v) = ((1+\epsilon/3)w(v) + \epsilon\gamma / 2r)\sqrt{\log t / \log d}\end{align*} If $w_i = w'(S_i) / |S_i|$ is the average weight on $S_i$, we have $\sum_i \sigma_i w_i \leq (1-\epsilon/6)\gamma(\parts,\admat)$. By the definition of $\gamma(\parts,\admat)$, there must be some $i,j$ such that $w_iw_j < D_{ij}$. Further, since $w_i\geq \epsilon\gamma /2r$ for all $i$, we must have $D_{ij} > \epsilon^2\gamma^2/4r^2$.
    
    By Markov's inequality, there are sets $A_i\subset S_i$ and $A_j\subset S_j$, of size at least $\epsilon |S_i| / 4$ and $\epsilon |S_j| / 4$ respectively such that for all $v\in A_i$, $w'(v) \leq (1+\epsilon/3)w(S_i)/|S_i|$; likewise for $A_j$. Replace $A_i$ with a subset so that $|A_i| = \ceil{\epsilon |S_i| / 4}$, and in the same fashion take $|A_j| = \ceil{\epsilon |S_j| /4}$.

    Recalling that for all $v\in A_i$ we have $w'(v)\leq (1+\epsilon/3)w_i$ and hence \\$w(v) \leq w_i\sqrt{\frac{\log d}{\log t}} - \epsilon\gamma /4r\leq (w_i- \epsilon\gamma /4r)\sqrt{\frac{\log d}{\log t}}$. In particular, for all $v\in A_i, v'\in A_j$ we have $t^{-w(v)w(v')}\geq d^{-D_{ij} + \epsilon^2\gamma^2/16r^2}$. Since $w$ satisfies the gamma inequality, this means there are fewer than $td^{D_{ij}- \epsilon^2\gamma^2/16r^2}$ edges between $A_i$ and $A_j$. Since $D_{ij} \geq \epsilon^2\gamma^2/4r^2$, this is less than $\frac14 \floor{td^{D_{ij}}}$.  We now compute the probability that such sets $A_i, A_j$ exist in a randomly chosen $H\in \mathcal{D}(\parts,\admat)(t,d)$. There are $r^2$ choices for $i$ and $j$, and at most $2^{2t}$ choices for $A_i$ and $A_j$ given such a choice of indices.
    
    We first consider the case $i\neq j$. Let $N = |S_i| |S_j|$, $M = |A_i||A_j|$ and $n = \floor{td^{D_{ij}}}$. The number of edges between $A_i$ and $A_j$ is distributed according to a hypergeometric distribution $HG(N,M,n)$, with mean $Mn/N$. By a result of Vatutin and Mikhailov \cite{hgisbinomial}, the hypergeometric distribution can be written as a sum of $n$ independent Bernoulli random variables, and hence the tail bounds from Lemma~\ref{L:chernoff} apply. Therefore, the probability there are fewer than $\epsilon^2 td^{D_{ij}}/64$ edges between $A_i$ and $A_j$ (which is less than half the mean by taking $t$ large) is at most $\exp(-(\epsilon^2/64)^2 td^{D_{ij}}/2) \leq \exp(-\epsilon^2td^{D_{ij}}/2^{13})$, and thus the expected number of pairs $(A_i,A_j)$ of this kind  in a randomly selected $H$ is at most $r^22^{2t}\exp(-\epsilon^2td^{D_{ij}}/2^{13})\leq \epsilon/2r^2$ provided we take $D$ sufficiently large.
    
     If instead $i=j$, we follow the same argument, but only find one set $A_i\subset S_i$, and consider the edges within $A_i$. This time there are at most $2^t$ choices, each of which occuring with small probability by a hypergeometric tail bound for a total probability of existence of at most $\epsilon/r^2$ again. The result follows taking a union bound.
\end{proof}

    \section{Proof of lower bounds}
    
    We start with the application to graphs with a weighting where each weight is on a reasonable proportion of the vertices. Our first aim in this section is to prove Theorem~\ref{T:extremalfunction}. We will instead prove the following generalisation of that theorem; Theorem~\ref{T:extremalfunction} follows immediately by taking an appropriate choice of $\epsilon,\gamma$ and using only the optimal value $p = 0.715\dotsc$. 
    
    \begin{restatable}{theorem}{gammalb2}
    \label{T:gammalb2}
    Let $\epsilon,\gamma >0$ and $0<p<1$. There is a constant $D = D_{\ref{T:gammalb2}}(\epsilon,\gamma,p)$ such that the following holds.
    
    Let $\parts$ be a weight vector, and $\admat$ a matrix with entries from $[-\infty,1]$ such that $\gamma(\parts,\admat)>\gamma$ and all $\sigma_i > \epsilon$. 
    Then for all $t>d>D$, there is a graph \\$G = G_{t,d,p,\gamma(\parts,\admat),\epsilon}$ (which does not depend on $\parts$ or $\admat$ except through $\gamma$)  with the following properties.
    
    \begin{itemize}
        \item $|G| \geq (1-2\epsilon)\gamma(\parts,\admat) t\sqrt{\log_{1/1-p}d}$
        \item $G$ has density at least $p-2\epsilon$
        \item All but at most $2^{-t}$ proportion of graphs from $\mathcal{D}(\parts,\admat)(t,d)$ are not minors of $G$.
    \end{itemize}
    \end{restatable}
    We will be building the graph for Theorem~\ref{T:gammalb2} as a blowup of a graph obtained from Lemma~\ref{L:basegraph}.
    \begin{definition}
    Let $G_0$ be a graph, and $k$ an integer. The \textsl{balanced blowup} $G_0(k)$ is the graph constructed from $|G_0|$ disjoint independent sets $I_v$ of size $k$  (each corresponding to a different vertex of $G_0$), and an edge between $x\in I_v,y\in I_w$ exactly when $v$ and $w$ are adjacent in $G_0$.
    \end{definition}
    What does a minor look like in such a blowup? Let $(V_h)$ be an $H$ model in $G_0(k)$. If, for some $h\in H, v\in G_0$, we use multiple vertices from $I_v$, we could also have a $H$ model deleting all but 1 such vertex. Considering only minimal models, we can assume $|V_h\cap I_v|\leq 1$ for all $h\in H, v\in G_0$. 
    
    This naturally lets us associate the parts $V_h$ with subsets of $G_0$, instead of $G_0(k)$. Having done so, each vertex of $G_0$ can appear in at most $k$ of these subsets. We call a collection of subsets of $V(G_0)$ where each vertex appears at most $k$ times a $\textsl{k-blobbing}$ in $G_0$.
    \begin{proof}[Proof of Theorem~\ref{T:gammalb2}]
        Let $G_0 = G_{d,p,\epsilon_1}$ be the graph from Lemma~\ref{L:basegraph}, applied with $\epsilon_1 = \epsilon^3\gamma^2 /64$ in place of $\epsilon$, and $d,p$ as in the theorem (we can do this provided $D > D_{\ref{L:basegraph}}(\epsilon_1,p)$).
        
        Let $l = \sqrt{\log_{1/1-p}d}$, and $ k = \floor{(1-\epsilon)\gamma(\parts,\admat)tl/d}$. We hereon assume that $D$, and hence $t,d,l,k$ are sufficiently that large all necessary inequalities hold; the exact requirements are suppressed. Let $G = G_0(k)$ be the balanced blowup of $G_0$. Then since $G_0$ has density at least $p$, $G$ has density at least $p-\epsilon$. As $\gamma tl/d > 2/\epsilon$, we have $|G| \geq (1-2\epsilon)\gamma tl$.
        
        We start with some notation. We will let $H\in \mathcal{D}(\sigma,\admat)(t,d)$ be an arbitrary element, and recall $H$ has vertex set $\{1,\dotsc,t\}$, and a vertex partition $S_1,\dotsc,S_r$ for some $r\geq 0$, where $|S_i| = s_i$ takes the value $\ceil{\sigma_i t}$ or $\floor{\sigma_i t}$, and there are exactly $\floor{td^{D_{ij}}}$ edges of $H$ between $S_i$ and $S_j$. We will show that if $H$ is a minor of $G$, this means $H$ has a particular structure --- and that when we later choose a random element $H\in \mathcal{D}(\parts,\admat)(t,d)$, this structure will be unlikely to occur. 
        
        If $H$ is a minor of $G$, this means there is some disjoint collection $(V_i)_{i=1}^t$ of disjoint connected non-empty subsets forming a model of $H$ in $G$. We will use the properties of $G_0$ to deduce certain pairs of parts cannot be adjacent in the model, and therefore cannot be edges of $H$.

        \begin{clm}
        For any collection $(V_i)_{i=1}^t$ of disjoint non-empty subsets of $V(G)$, there exist indices $a\leq b$ such that the density of pairs $(V_i,V_j)$ where $i\in S_a, j\in S_b$ and $V_i$ is (distinct from and) non-adjacent to $V_j$ is at least $\rho_{a,b} = \epsilon^2 d^{\epsilon_1 - D_{a,b}}2^{-13}$.
        \end{clm}
        
       Before proving this claim, we will show how it implies the theorem. Let $H$ now be a randomly chosen element of $\mathcal{D}(\parts,\admat)(t,d)$. If $H$ is to be a minor of $G$, there must be some model $(V_h)_{h\in [t]}$ of $H$ in $G$. Given this model, our claim gives indices $a,b$ as above. We take $D$ sufficiently large that $\epsilon^2D^{\epsilon_1/2}2^{-13} > 1$, so we may assume $D_{a,b} > \epsilon_1/2$ as else $\rho > 1$ which cannot be the case since it is a density.
        
        This gives us some structural information about $H$: all of the edges between $S_a$ and $S_b$ avoid the $\rho_{a,b}$ proportion of non-adjacent pairs --- we would like to show this is unlikely. Since $H$ is chosen randomly, this means the edges between $S_a$ and $S_b$ form a uniformly random $\floor{td^{D_{a,b}}}$ subset. We now break into two cases for the analysis, and firstly assume $a \neq b$.
        
        In this case, we are picking $\floor{td^{D_{a,b}}}\geq \frac12 td^{D_{a,b}}$ edges. The probability we avoid the $\rho_{a,b}$ proportion can then be bounded as follows.
        
        \begin{align*}
            \mathbb{P}(\textrm{we avoid these bad edges})\leq \binom{(1-\rho_{a,b})|S_a||S_b|}{\floor{td^{D_{a,b}}}}/\binom{|S_a||S_b|}{\floor{td^{D_{a,b}}}} \leq (1-\rho_{a,b})^{\frac12 td^{D_{a,b}}} \\\leq \exp(-\frac12td^{D_{a,b}}\rho_{a,b})\leq \exp(-\epsilon^2td^{\epsilon_1}/2^{14})
        \end{align*}
        
        If instead $a = b$, we can perform the analagous calculation for edges within $S_a$, and obtain the following equivalent bound.
        \begin{align*}
                    \mathbb{P}(\textrm{we avoid these bad edges})\leq \binom{(1-\rho_{a,b})\binom{|S_a|}{2}}{\floor{td^{D_{a,b}}}}/\binom{\binom{|S_a|}{2}}{\floor{td^{D_{a,b}}}}\leq\exp(-\epsilon^2td^{\epsilon_1}/2^{14})
        \end{align*}
        
        Therefore, for a random graph $H$, we have bounded the probability a fixed collection of subsets forms a model of $H$ in $G$ by this final quantity. If we can also show there are not too many collections of subsets which could form a model, we will be able to show our result. In fact, it suffices to consider only minimal models; if a graph does not have any minimal models, it will have no models at all. Using the discussion before our theorem, we will associate a minimal model $(V_h)$ in $G$ with a $k-$blobbing in $G_0$. We can now use an encoding argument.
        
         Start by listing (in order) $kd$ vertices of $G_0$; these are the vertices that will be allowed to be used in the blobbing (and note since every blobbing has total size at most $kd$, this is valid). Next, take a $\{0,1\}$ valued sequence of length $kd$, where the first 1 denotes the index where we start listing $V_1$, the second where we have finished listing $V_1$ at the previous index and started listing $V_2$, and finally the $t+1^{st}$ 1 denoting where we stopped listing $V_t$ at the previous index. Any duplicate vertices are ignored.  This gives an (injective) way to encode any given blobbing, and it is easily seen there are at most $(2d)^{kd}$ such encodings, hence at most that many blobbings.

       We can now take $D$ large enough that $(2d)^{kd}\exp(-\epsilon^2d^{\epsilon_1}t/2^{14}) \leq 2^{-t}$. Having done so, except with probability at most $2^{-t}$ a randomly chosen $H$ will not have any minimal model in $G_0(k)$,  and so $H$ will not be a minor of $G$. Therefore, it suffices to prove Claim 1 to prove our theorem.
       
        \begin{proof}[Proof of Claim 1]
        
        We start by replacing each $V_i$ with the corresponding projection to $V(G_0)$; this does not affect adjacency, and our subsets now form a $k$-blobbing in $G_0$.
        
        Recall that $H$ is equipped with a vertex partition $(S_i)_{i=1}^r$, where $|S_i| = \floor{\sigma_i t}$ or $\ceil{\sigma_i t}$. Suppose that $X_i$ is the (ordered) collection of subsets $(V_{\sum_{j=1}^{i-1}s_i + 1},\dotsc,V_{\sum_{j=1}^i s_i}) $ corresponding to $S_i$, and let $b_i$ denote the average size of these $|S_i|$ subsets. Since the subsets $(V_i)_{i=1}^t$ form a $k$-blobbing in $G_0$, $\sum_i |S_i|b_i\leq dk\leq (1-\epsilon)\gamma tl$. 
        In particular, $\sum_i \sigma_i tb_i < (1-2\epsilon/3)\gamma tl$ - since all $\sigma_i t > \gamma t > \epsilon D$, the replacing $\floor{\sigma_i t}$ with $\sigma_i t$ cannot increase the sum by a large factor.
        
        Let $\beta_i = b_i/l + \epsilon\gamma /3$, and so $\sum_i\sigma_i(\beta_i/(1-\epsilon/3)) \leq\gamma(\parts,\admat)$.
        Recall $\gamma(\parts,\admat)$ is the minimum of $\sum \sigma_i\beta'_i$ subject to all $\beta'_i \beta'_j \geq D_{ij}$, and $\beta'_i > 0$. Therefore, there must be some indices $a,b$ such that $\beta_a\beta_b < D_{ab}(1-\epsilon/3)^2\leq D_{ab}(1-\epsilon/2)$, and in particular $b_ab_b\leq D_{ab}(1-\epsilon/2)l^2$. Since $\beta_a,\beta_b > \epsilon\gamma /3$, we must have $D_{ab}>\epsilon^2\gamma^2/9$. 
        
        By Markov's inequality, the subsets $V_w$ corresponding to at least $\epsilon/16$ proportion of $w\in S_a$ have size at most $(1+\epsilon/8)b_a$, and call the set of such $w$ $Y_a$, where we must have $|Y_a|\geq \epsilon|S_a|/16$. We analagously define $Y_b$.
        
        Suppose that $A = V_{v}$ is a part corresponding to some $v\in Y_a\subset V(H)$, and $B = V_{w}$ likewise for $w\in Y_b$. Let $\rho = d^{\epsilon_1-D_{ab}}$, and let $Y_a^*$ be the set of vertices $i\in Y_a$ such that the corresponding part $V_i\subset V(G_0)$ is disjoint from and not adjacent to at most $\epsilon |S_b|\rho/2^8$ parts corresponding to a vertex in $Y_b$. In particular, elements of $Y_a^*$ must be non-empty.
        
        Suppose first that $|Y_a^*|\geq \epsilon|S_a|/32$, and similarly for $Y_b^*$. We will build up families $(A_i)_{i=1}^{n_a}$ and $(B_j)_{j=1}^{n_b}$ of disjoint sets from $Y_a^*$, $Y_b^*$ respectively, one element at a time and in a balanced fashion, such that there is at most a density $\rho/2$ of disjoint pairs between them. We would like to attain $n_a = n_b \geq d^{D_{ab}-\epsilon_1}$. Suppose we have not yet done so, and thus reordering if necessary $n_a\leq n_b\leq d^{D_{ab}-\epsilon_1}$. We would like to find another part corresponding to a vertex in $Y_a*$ that we can add to the family $(A_i)$, and therefore increase $n_a + n_b$ to get closer to our goal.
        
        For $A$ a part (i.e. some subset $V_v$) corresponding to a vertex $v \in Y_a* \subset V(H)$, and $B$ likewise for some $w\in Y_b^*$, we have $1\leq|A||B|\leq l^2$, and so $|A|\leq l^2$. This means each already chosen part intersects at most $kl^2$ other parts, and in particular either $n_a \geq d^{D_{ab}-\epsilon_1}\geq d^{D_{ab}(1-\epsilon/8) + \epsilon_1}$, or there are at least $\frac{\epsilon}{32}\sigma_at - d^{D_{ab}-\epsilon_1}kl^2 \geq \frac{\epsilon}{64}|S_a|$ parts which correspond to a vertex in $Y_a^*$, are so far unchosen, and are disjoint from all already chosen sets. We remark that this is the only place where we use the lower bound on $\sigma_a$ (except through the implicit upper bound on $r$).
        
        If all of these parts are non-adjacent to at least $\rho n_b/2$ parts corresponding to vertices from $Y_b^*$, by averaging some part $V_w: w\in Y_b^*$ is disjoint from and non-adjacent to at least $\epsilon\rho/128 |S_a|$ sets from $Y_a^*$. This is a contradiction, since then $w$ would not be in $Y_b^*$! Hence some choice of part $A$ is adjacent to at most $\rho n_b/2$ such parts ,and we can therefore add it to our family $(A_i)$. We continue in this fashion to build up such families with $n_a = n_b \geq d^{D_{ab}-\epsilon_1/8}$. But this now contradicts the definition of $G_0$! Hence our assumption $|Y_a^*|\geq \epsilon|S_a|/32, |Y_b^*|\geq \epsilon|S_b|/32$ must have been incorrect.
        
        We therefore have (reordering if necessary) that $|Y_a\setminus Y_a^*|\geq \epsilon|S_a|/32$. In particular, there are at least $\epsilon^2\rho |S_a||S_b|/2^{13}$ pairs $v\in Y_a\setminus Y_a^*$, $w\in Y_b$ such that $V_v$ and $V_w$ are disjoint and non-adjacent, which finishes the proof of our claim and thus the theorem.
        \end{proof}
    \end{proof}
    
    We remark that our graphs (for fixed $p,d,\epsilon_1$) are different sized blowups of the same base (pseudorandom) graph, and so will look very similar. It is interesting that the random structure required is universal.

We now move onto the lower bound for complete bipartite graphs. For $H = K_{ft/\log t,t}$, we have $\gamma_2(H) \leq 2\sqrt{f/\log t} = o(1)$ and hence we cannot apply the previous theorem --- we have $\sigma_1 < \epsilon$. However, we are only looking for a complete bipartite graph, and so even a single non-edge will suffice for our purposes, rather than the $\epsilon^2\rho$ proportion required earlier. This allows us to mimic the proof of the previous theorem in this case --- we leave it as a sketch, with the details able to be filled in as in Theorem~\ref{T:gammalb2}.

\lowerboundbipartite*

\begin{proof}[Proof Sketch]

    Let $d = ft/\log t$ (so $H$ has $td$ edges), and let $G_0$ be the graph $G_{d,p,\epsilon}$ from Lemma~\ref{L:basegraph}, with $p = 0.715\dotsc$. Let $G = G_0(k)$, where\\$k=\floor{2(1-\epsilon)\log t\sqrt{1/f\log(1/1-p)}}$. This $G$ has at least $2(1-2\epsilon)t\sqrt{f/\log(1/1-p)}$ vertices, and density at least $p-\epsilon$ --- hence average degree at least $2(\alpha - 3\epsilon)t\sqrt{f}$. We again neglect rounding throughout the proof by taking $T$ sufficiently large.
    
    Suppose we have an $H$ model in $G$, with subsets $X_1 = \{A_1,\dotsc,A_{ft/\log t}\}$ representing the left hand class. and $X_2 = \{B_1,\dotsc,B_t\}$ the right hand. If $b_i$ is the average size of the sets in $X_i$, by the definition of $\gamma_2$, we must have $b_1b_2\leq (1-\epsilon)^2\log_{1/1-p}t$. Further, at least $\epsilon/8$ proportion of sets (from each of $X_1$ and $X_2$) have size at most $(1+\epsilon/8)$ times the average size by Markov's inequality. Call these sets $Y_1$, $Y_2$ and note that if $A_i\in Y_1, B_j\in Y_2$ then $|A_i||B_j|\leq (1-\epsilon)\log_{1/1-p}d$.
    
    We note that these sets are sufficiently small that the results of Lemma~\ref{L:basegraph} will guarantee that for any $d^{1-\epsilon}$ sets representing the left hand class, and the same for the right hand class which are pairwise disjoint, then some non-zero density (and in particular, at least 1) of the pairs is non-adjacent.
    
    How can we find these disjoint sets? Each set must have size at most $\log t / \sqrt{f\log(1/1-p)}$, since we can assume all sets under consideration are non-empty, and so intersects at most $O((\log t)^2 / f)$ other sets. It follows we can greedily find sufficiently many (of order $t^{1-\epsilon}$) disjoint sets to apply Lemma~\ref{L:basegraph}. This contradicts that $(A_i),(B_j)$ was a $H$ model, since if a positive density of pairs are non-adjacent in particular at least one such pair is.
    
\end{proof}
\subsection*{Acknowledgements}
The author was supported by an EPSRC DTP Studentship. The author would also like to thank Andrew Thomason for many helpful discussions.
\bibliographystyle{plain}
\bibliography{biblio}

\end{document}